\documentclass[12pt, reqno]{amsart}
\usepackage{hyperref}
\usepackage{todonotes}
\usepackage{graphics}
\usepackage{graphicx}
\usepackage{enumerate}
\usepackage{mathtools}
\usepackage{todonotes}
\usepackage{comment}
\usepackage{amssymb}
\usepackage{amsmath}
\usepackage{amsthm}
\usepackage{multicol}
\usepackage[english]{babel}
\usepackage{cmap}
\usepackage{url}
\usepackage{setspace}

\usepackage{setspace}
\makeatletter
\def\tbcaption{\def\@captype{table}\caption}
\makeatother
\newtheorem{theorem}{Theorem}
\newtheorem{lemma}[theorem]{Lemma}

\theoremstyle{definition}
\newtheorem{definition}{Definition}

\newtheorem{proposition}[theorem]{Proposition}

\newtheorem*{lemma*}{Lemma}

\theoremstyle{remark}

\newcommand{\beq}[1]{\begin{equation}\label{#1}}
\newcommand{\eeq}{\end{equation}}

\newcommand{\abs}[1]{\left\vert#1\right\vert}

\begin{document}

\title[Thabit and Williams Numbers as a Sum of Two Repdigits]{Thabit and Williams Numbers Base $b$ as a Sum or Difference of Two $g$-Repdigits}

\author[K. N. Ad\'edji]{Kou\`essi Norbert Ad\'edji}
\address{Institut de Math\'ematiques et de Sciences Physiques, Universit\'e D'A\-bo\-mey-\-Ca\-la\-vi, Bénin}
\email{adedjnorb1988@gmail.com}

\author[M. Bliznac Trebje\v{s}anin]{Marija {Bliznac Trebje\v{s}anin}$^1$}
\thanks{$^1$ Corresponding author.}
\address{University of Split, Faculty of Science, Ru\dj{}era Bo\v{s}kovi\'{c}a 33,
21000 Split, Croatia}
\email{marbli@pmfst.hr}

\author[J. Ple\v{s}tina]{Jelena {Ple\v{s}tina}}
\address{University of Split, Faculty of Science, Ru\dj{}era Bo\v{s}kovi\'{c}a 33,
21000 Split, Croatia}
\email{jplestina@pmfst.hr}

\begin{abstract}
We investigate cases where Thabit and Williams numbers in base $b$ can be expressed as the sum or difference of two $g$-repdigits. For specific values of $b$ and $g$, we describe parametric solutions yielding infinitely many solutions for some equations and establish upper bounds for the parameters of the remaining finitely many solutions. As an illustration, we also provide a complete solution for some equations.
\end{abstract}

\maketitle 

\noindent{\it 2020 {Mathematics Subject Classification:}} 11D61, 11J68, 11J86

\noindent{\it Keywords}: repdigit, linear form in logarithms, Thabit number, reduction method

\section{Introduction}

A Thabit number, or a Th\^abit ibn Qurra number, or $321$
number is an integer of the form $3\cdot 2^n-1$ for a non-negative integer $n$. The sequence of Thabit numbers begins with
$$2, 5, 11, 23, 47, 95, 191, 383, 767, 1535, 3071, 6143,\dots$$
We can generalize this concept in multiple ways. Let $b\geq 2$ be an integer. A Thabit number base $b$ is a number of the form
$(b+1)b^n-1$ for a non-negative integer $n$. A second kind Thabit number base $b$ is a number of the form $(b + 1)b^n + 1$ for a non-negative integer $n$. Another natural generalization are the Williams numbers.
For an integer $b\geq 2$, a Williams number base $b$ is a number of the form $(b-1)b^n\pm 1$ for a non-negative integer $n$. The question of which Thabit or Williams number is prime is a widely studied problem. In this paper, we will consider another property of Thabit and Williams numbers.

\begin{definition}\label{def:repdigit} Let $g\geq 2$ be an integer. A positive integer $N$ is called a repdigit in base $g$ or simply a $g$-repdigit if all of the digits in its base $g$ expansion are equal.

If $N$ is a $g$-repdigit then it has the form
$$N=a\left(\frac{g^m-1}{g-1}\right),\quad \textrm{for }m\geq1,\ a\in\{1,2,\dots,g-1\}.$$
Taking $g=10$, a positive integer $N$ is simply called a repdigit.
\end{definition}
 The aim of this work is to prove when a Thabit number base $b$ or a Williams number  base $b$ is a sum or a difference of two $g$-repdigits. More precisely, let $b\geq 2$ and $g\geq 2$ be integers. We are searching for solutions $(d_1,d_2,l,m,n)$ to each Diophantine equation
\begin{equation}\label{eq:first_sum}
    (b\pm1)b^n\pm 1=d_1\left(\frac{g^l-1}{g-1}\right)+d_2\left(\frac{g^m-1}{g-1}\right),
\end{equation}
and 
\begin{equation}\label{eq:second_diff}
    (b\pm1)b^n\pm 1=d_1\left(\frac{g^l-1}{g-1}\right)-d_2\left(\frac{g^m-1}{g-1}\right),
\end{equation}
where $1\leq d_1,d_2\leq g-1$, $n\geq0$ and $l,m\geq1$ are integers. 

\medskip

For some choices $b$ and $g$, some of the equations from \eqref{eq:first_sum} have infinitely many solutions $(d_1,d_2,l,m,n)$. The statement of this proposition is easily verified.
\begin{proposition}\label{prop:infinitely_sum}
\mbox{}
\begin{enumerate}[a)]
    \item
    Let $b=2$ and $g=2^k$, for an integer $k\geq 2$. Then $(d_1,d_2,l,m,n)=(2,g-1,1,m,mk)$, $m\in\mathbb{N}$, represent infinitely many solutions to the equation 
      \begin{equation}\label{eq:sum_mp}
      (b-1)b^n + 1=d_1\left(\frac{g^l-1}{g-1}\right)+d_2\left(\frac{g^m-1}{g-1}\right).
  \end{equation}
\item Let $b=2$ and $g=2^k$, for an integer $k\geq 2$.  Then $(d_1,d_2,l,m,n)=(d_1,g-1-d_1,m,m,mk)$, $d_1\in\{1,2,\dots,g-2\}$, $m\in\mathbb{N}$, represent infinitely many solutions to the equation 
      \begin{equation}\label{eq:sum_mm}
      (b-1)b^n - 1=d_1\left(\frac{g^l-1}{g-1}\right)+d_2\left(\frac{g^m-1}{g-1}\right).
  \end{equation}

    \item Let $b = a^\beta$ and $g = a^\alpha$, where $a, \alpha, \beta$ are positive integers such that $a \geq 2$ and either $\alpha > \beta$ with $\gcd(\alpha, \beta) = 1$, or $\alpha = \beta$ and $b \neq 2$. 
    For each integer $k$ satisfying $0 \leq k \leq \alpha - \beta$ (where $k > 0$ if $b = 2$),  the quintuple $(d_1, d_2, l, m, n) = (a^\alpha - a^k(a^\beta - 1), a^k(a^\beta - 1) - 1, l, l+1, n)$ represents infinitely many solutions to the equation \eqref{eq:sum_mm}, where $n \geq l \geq 1$ are the solutions to the linear Diophantine equation $\beta n - \alpha l = k$.

    \item Let $b=a^\beta$ and $g=a^\alpha$, where $a, \alpha, \beta$ are positive integers such that $a \geq 2$ and $\alpha > \beta$ with $\gcd(\alpha, \beta) = 1$. 
    For each integer $k$ satisfying $0 \leq k < \alpha - \beta$, the quintuple $(d_1, d_2, l, m, n) = (a^{\alpha} - a^k(a^{\beta} + 1), a^k(a^{\beta} + 1) - 1, l, l+1, n)$ represents infinitely many solutions to the equation:
    \begin{equation}\label{eq:sum_pm}
        (b + 1)b^n - 1 = d_1 \left( \frac{g^l - 1}{g - 1} \right) + d_2 \left( \frac{g^m - 1}{g - 1} \right),
    \end{equation}
    where $n > l \geq 1$ are solutions to the linear Diophantine equation $\beta n - \alpha l = k$.

\item Let $b=2^{k}$ and $g=2^{k+1}$, for an integer $k\geq 1$.  
Then $(d_1,d_2,l,m,n)=(g-2,1,k(t+1)-1,k(t+1)+1,(k+1)t+k)$, $t\in\mathbb{N}$, represent infinitely many solutions to the equation \eqref{eq:sum_pm}.
 
\end{enumerate}
\end{proposition}

For $b=2$, one of the equations \eqref{eq:second_diff}, also has infinitely many solutions.

\begin{proposition}\label{prop:infinitely_diff}
     Let $b=2$ and $g\geq 2$ be any integer. Then $(d_1,d_2,l,m,n)=(d_1,d_1,l,l,0)$, for $l\in\mathbb{N}$, $1\leq d_1\leq g-1$ represents infinitely many solutions of the equation 
\begin{equation}\label{eq:diff_mm}
      (b-1)b^n - 1=d_1\left(\frac{g^l-1}{g-1}\right)-d_2\left(\frac{g^m-1}{g-1}\right).
  \end{equation}
\end{proposition}

 The main results about \eqref{eq:first_sum} and \eqref{eq:second_diff} are given in the following theorems.
\begin{theorem}\label{tm:first}
Let $b, g \geq 2$ be integers. Then, except for the cases described in Proposition \ref{prop:infinitely_sum}, the Diophantine equations \eqref{eq:first_sum} have finitely many solutions in integers $1 \leq d_1, d_2 \leq g-1, \, n \geq 0$ and $ l,m \geq 1$. Moreover, 
$$l, m < 2.71 \cdot 10^{29} \log^2{g} \log^3{b} \log^2{ \left( \max{\{ b,g \}} \right) } $$
and
 $$n < 2.08 \cdot 10^{29} \log^3{g} \log^2{b} \log^2{ \left( \max{\{ b,g \}} \right) } .$$
\end{theorem}

\begin{theorem}\label{tm:second}
Let $b, g \geq 2$ be integers. Then, except for the cases described in Proposition \ref{prop:infinitely_diff},  the Diophantine equations \eqref{eq:second_diff} have finitely many solutions in integers $1 \leq d_1, d_2 \leq g-1, \, n \geq 0$ and $ l,m \geq 1$. Moreover, 
$$m \leq l < 2.85 \cdot 10^{29} \log^2{g} \log^3{b} \log^2{ \left( \max{\{ b,g \}} \right) } $$
and
 $$n < 2.19 \cdot 10^{29} \log^3{g} \log^2{b} \log^2{ \left( \max{\{ b,g \}} \right) } .$$
\end{theorem}    

The outline of the paper is as follows. We begin by presenting some preliminary lemmas and results that will be used throughout the paper. Additionally, we provide a complete solution of equations \eqref{eq:first_sum} and \eqref{eq:second_diff} in the case  $g = 2$, as it is straightforward.

In Section 3, we examine equations involving the sum of two $g$-repdigits by deriving several inequalities between the parameters $l$, $m$, and $n$. We then analyze two linear forms in logarithms to obtain upper bounds for the same indices in terms of $b$ and $g$. Since there exist infinitely many solutions for certain values of $b$ and $g$ to equations \eqref{eq:first_sum}, the proof of Lemma \ref{lemma:lambda_2_neq_0} requires the analysis of several subcases and the use of multiple distinct techniques. In particular, it involves solving an additional Diophantine equation and developing tailored arguments to identify all possible solutions. This makes the proof significantly more intricate compared to similar results. In contrast, in Section 4, where we study the difference of two $g$-repdigits, more direct arguments suffice to prove an analogous result.

To fully resolve certain equations, we describe all possible solutions in the decimal case $g = 10$, and for $2 \leq b \leq 12$, in Subsections~3.1 and~4.1. Note that the approach is based on the reduction method by Bravo, Gómez, and Luca \cite{BGF}  using some auxiliary tools from \cite{Dujella_book}.

The motivation for this paper arises from recent work on repdigits, such as~\cite{alt_pell_repdigits}, \cite{AFT:2023}, \cite{ddam}, and~\cite{erd_kes_lucas_repdigits}. Our methods and overall approach are similar to those used in other studies on number sequences and their interesting representations, such as~\cite{abt}, \cite{nat1}, and~\cite{BGF}.

\section{Preliminaries}
First, we will list some supporting definitions and results. 

Let $\eta$ be an algebraic number of degree $t,$ let $a_0 \ne 0$ be the leading coefficient of its minimal polynomial over $\mathbb{Z}$ and let $\eta=\eta^{(1)},\ldots,\eta^{(t)}$ denote its conjugates. The logarithmic height of $\eta$ is defined by
\[
 h(\eta)= \frac{1}{t}\left(\log |a_0|+\sum_{j=1}^{t}\log\max\left(1,\left|\eta^{(j)} \right| \right) \right).
\]
If $p$ and $q$ are integers such that $q>0$ and $\gcd (p, q)=1,$ then for $\eta=p/q$ the above definition reduces to $h(\eta)=\log(\max\{|p|,q\}).$ 

We recall the following properties of the logarithmic height (Property 3.3 of \cite{Wald2000}): \\
If $\eta_1$ and $\eta_2$ are algebraic numbers, then we have
\begin{align*}
 h(\eta_1\eta_2) &\leq h(\eta_1) + h(\eta_2), \\
 h(\eta_1 \pm \eta_2) &\leq h(\eta_1)+ h(\eta_2) +\log2.
\end{align*}
If $\eta_1 \neq 0$ is a nonzero algebraic number and $j\in \mathbb{Z}$, then we have
\begin{align*}
h(\eta_1^j)&=|j|h(\eta_1).
\end{align*}

\begin{lemma}[Theorem~9.4 of \cite{BMS:2006}]\label{tm:BMS}
Let $\gamma_1,\dots,\gamma_s$ be real algebraic numbers and let $b_1,\dots,b_s$ be nonzero integers. Let $D$ be the degree of the number field $\mathbb{Q}(\gamma_1,\dots,\gamma_s)$ over $\mathbb{Q}$ and let $A_j$ be a positive real number satisfying  
$$A_j\geq\max\{Dh(\gamma_j),|\log\gamma_j|,0.16\},\quad \textrm{for }j=1,\dots,s. $$ 
Assume that 
$$B\geq \max\{|b_1|,\dots,|b_s|\}.$$
If $\Lambda\coloneqq\gamma_1^{b_1}\cdots\gamma_s^{b_s}-1\neq 0$, then
\begin{equation}\label{ineq:BMS}
|\Lambda|\geq \exp(-1.4\cdot30^{s+3}\cdot s^{4.5}\cdot D^2(1+\log D)(1+\log B)A_1\cdots A_s).
\end{equation}
\end{lemma}
The next auxiliary result from \cite{SGL} will be useful for transforming inequalities.
\begin{lemma}[Lemma~7 of \cite{SGL}]\label{lemma:supporting}
If $ \ell \geq 1$, $H>(4 \ell^2)^\ell$ and $H>L/(\log L)^\ell$, then
$$L<2^\ell H(\log H)^\ell.$$
\end{lemma}

As one step of the proof for the application, we will apply the reduction method originally introduced by Baker and Davenport \cite{bd}. The following is a variation of the result of Dujella and Peth\H{o} (see
\cite[Lemma 5]{Dujella-Peto}). In this form, which we need for our purposes, the proof of the first part of the lemma is given by Bravo, Gómez and Luca (see \cite[Lemma~1]{BGF}).
For the second part of the lemma, which we will use when $\varepsilon < 0$, we can follow the argument explained in \cite[Remark~14.3]{Dujella_book} and derive the conclusion from an inequality of the form \eqref{eqn:baker_d}.

For a real number $x,$ we write $\left\Vert x\right\Vert$ for the distance from $x$ to the nearest integer.
\begin{lemma}\label{lemma:reduction}
Let $M$ be a positive integer, let $p/q$ be a convergent of the continued fraction of the irrational $\tau$ such that $q>6M$, and let $A,B,\mu$ be some real numbers with $A>0$ and $B>1$. Let 
$$
\varepsilon=\|\mu q\|-M\cdot \|\tau q\|.
$$
\begin{enumerate}[a)]
\item If $\varepsilon>0$, then there is no solution to the inequality
\begin{equation}\label{eqn:baker_d}
0<|m\tau -n+\mu|<AB^{-w},
\end{equation}
in positive integers $m,n$ and $w$ with 
$$m\leq M\quad\textrm{and}\quad w\geq \frac{\log(Aq/\varepsilon)}{\log B}.$$
\item 
Let $r = \lfloor \mu q + \tfrac{1}{2} \rfloor$. 
The inequality~\eqref{eqn:baker_d} admits a solution in positive integers $m, n, w$ satisfying
$$m \le M \quad \text{and} \quad w > \frac{\log(3Aq)}{\log B},$$
only if the smallest positive integer $m_0$ satisfying the linear congruence
$$mp \equiv -r \pmod{q}$$
also satisfies $m_0 \leq M$. In that case, $m = m_0$.
\end{enumerate}
\end{lemma}
\begin{proof} We will prove statement b).
    Let $r=\lfloor{\mu q+1/2}\rfloor$. Then $|\mu q-r|\leq 1/2$. Since $p/q$ is a convergent of $\tau$ then $|p-q\tau|<1/q$. 

    After multiplying \eqref{eqn:baker_d} with $q$ and rearranging, we have 
    \begin{equation}
        |q\mu+mp-qn|<qAB^{-w}+m|p-q\tau|.
    \end{equation}
    Now observe, 
    \begin{align*}
        |mp-qn+r|&=|mp-qn+q\mu-q\mu+r|\leq |mp-qn+q\mu|+|q\mu-r|\\
        &<qAB^{-w}+m|p-q\tau|+1/2<qAB^{-w}+m/q+1/2<qAB^{-w}+2/3
    \end{align*}
    where we have used the fact that $q>6M>m$.

If $qAB^{-w}\geq 1/3$ then $w\leq (\log(3Aq))/(\log B).$

 If $qAB^{-w}< 1/3$, then $mp-qn+r=0$. Consequently, $mp+r\equiv 0\ (\bmod \ q)$ must hold, and there is a unique integer solution $0\leq m_0<q$. If $m_0> M$ then there is no solution to \eqref{eqn:baker_d} in this case. If $m_0\leq M$, then $m=m_0$.
\end{proof}

As $g = 2$ implies $d_1 = d_2 = 1$, this case merits individual consideration.
\begin{lemma}\label{lemma:g_2}
    If $g=2$ then a solution for \eqref{eq:first_sum} or \eqref{eq:second_diff} exists only for $b$'s of the form $b=(2^l-1)\pm (2^m-1)\pm 1\pm 1$, where $m$ and $l$ are some positive integers, and then $n=0$.
\end{lemma}\begin{proof}
 Let $g=2$. Then $d_1 = d_2 = 1$, so we observe equations
\begin{equation}\label{eq:lem_1}
(b \pm 1)b^n \pm 1 = (2^l - 1) \pm (2^m - 1).
\end{equation}
For $n \geq 1$ the left-hand side of \eqref{eq:lem_1} is an odd integer, while the right-hand side is an even integer, which is a contradiction. Therefore, $n=0$ and the statement holds.
\end{proof}

\section{\texorpdfstring{The sum of two $g$-repdigits}{The sum of two g-repdigits}}

This section is devoted to proving Theorem \ref{tm:first}, so we consider equations \eqref{eq:first_sum}. Given that Lemma \ref{lemma:g_2} holds, we will assume $g \geq 3$ from now on. 
Note that for the equations \eqref{eq:first_sum}, without loss of generality, we can further assume $l \leq m$.

\begin{lemma}\label{lemma:g_3} 
All solutions to the equations \eqref{eq:first_sum} satisfy 
$$n < 2.5 m \log{g}.$$ 
\end{lemma} 
\begin{proof} 
From equations \eqref{eq:first_sum} and $d_1, d_2 \leq g-1, \, l\leq m$, we obtain \begin{equation}\label{eq:lem_g_3_0} \left( b \pm 1 \right) b^n < 2 g ^m . \end{equation} 
Given that $ 2 \leq b $ and $ 3 \leq g $, we take the logarithm of \eqref{eq:lem_g_3_0} and, after performing simple calculations, we obtain 
\begin{equation}\label{eq:lem_g_3} n \log{b} < \log{2} - \log{\left( b \pm 1 \right)} + m \log{g}. \end{equation} 
Let us divide the proof into two cases, depending on the value ${\log{2} - \log{\left( b \pm 1 \right)}}$. 
\begin{enumerate}[(i)] 
\item {For $b=2$ in the case ${\log{2} - \log{\left( b - 1 \right)}}$, we have
\begin{equation}\label{eq:lem_g_3_1} n \log{2} < \log{2} - \log{\left( 2-1 \right)} + m \log{g}. \end{equation} 
Dividing \eqref{eq:lem_g_3_1} by $\log{2}$, we get $n < 1 + m \frac{\log{g}}{\log{2}}$. 
Since $1 < \log{3} \leq \log{g}, \, 1 \leq m$ and $\frac{1}{\log{2}}<1.5 $, we get $n< 2.5 m \log{g}$.} 
\item {Let us examine the remaining cases. Since $\log{2} < \log{\left(b + 1\right)} $ for every $b \geq 2$, and $\log{2} \leq \log{\left(b - 1\right)} $ for every $b \geq 3$, from \eqref{eq:lem_g_3}, it follows that $ n < m \frac{\log{g}}{\log{b}}$. For $b \geq 2$, since $\frac{1}{\log{b}} < 1.5$ we get $n < 1.5 m \log{g} < 2.5m \log{g}.$ For $b \geq 3$, since $\frac{1}{\log{b}} < 0.92$ we get $n < 0.92 m \log{g} < 2.5m \log{g}.$}\qedhere \end{enumerate} 
\end{proof}

\begin{lemma}\label{lemma:g_4}
    All solutions to the equations \eqref{eq:first_sum} satisfy $$ l\leq m < 1.3 \left( n+ 1.6 \right) \frac{\log{b}}{\log{g}}+1.$$
\end{lemma}
\begin{proof} 
The statement obviously holds for $m=1$.
Let us assume that $m \geq 2$.
From \eqref{eq:first_sum}, and since $\left( b+1 \right)b^n + 1 < \left( \left( b+1 \right)b^n \right)^{1.3}$ for all 
$b \geq 2$ and $n \geq 0$, it follows that
\begin{equation}\label{eq:lem_g_4_1}
   g^{m-1} < \left( \left( b+1 \right)b^n \right)^{1.3}.
\end{equation} 
Taking the logarithm of \eqref{eq:lem_g_4_1}, and since $\log{\left(b+1\right)} < 1.6 \log{b}$ for all $b \geq 2$, we get the statement. Due to the assumption $l \leq m$, the statement also holds for $l$.
\end{proof}
Now we are ready to prove Theorem \ref{tm:first}.

\smallskip 

\textit{Proof of the Theorem \ref{tm:first}}
We will divide the proof into two parts, applying Lemma \ref{tm:BMS} in each.

\textit{Step 1.}
\\
Multiplying the equations \eqref{eq:first_sum} by $g-1$ and performing simple rearrangements, we obtain
\begin{equation}\label{lemma:g_5_1}
   \left( b \pm 1 \right) b^n \left( g-1\right) - d_2 g^m = \mp \left(g-1\right) + d_1 g^l - \left( d_1 + d_2 \right) .
\end{equation}
Taking the absolute value of \eqref{lemma:g_5_1} and using assumptions $d_1, d_2 \leq g-1$, $ g \geq 3$ and $l \geq 1$, results in
\begin{equation}\label{lemma:g_5_2}
  \abs{ \left( b \pm 1 \right) b^n \left( g-1\right) - d_2 g^m} < g^{l+2}. 
\end{equation}
Dividing the inequality \eqref{lemma:g_5_2} by $d_2 g^m$ and using $d_2 \geq 1$, we get
\begin{equation}\label{lemma:g_5_3}
\abs{ \frac{\left( b \pm 1 \right) b^n \left( g-1\right)}{d_2 g^m} -1} < \frac{1}{g^{m-l-2}}.
\end{equation}

We define 
\begin{equation}\label{lemma:g_5_3_0}
\Lambda_1 \coloneqq \frac{\left( b \pm 1 \right) b^n \left( g-1\right)}{d_2 g^m} -1 .
\end{equation}

Lemma \ref{tm:BMS} can be applied if $\Lambda_1 \neq 0$.

\begin{lemma}\label{lemma:lambda_1_0} If $\Lambda_1 = 0$, then $d_1=2$, $d_2=g-1$, $l=1$ and
 \begin{equation}\label{eq:lemma_lambda_1}
      (b\pm1)b^n + 1=d_1\left(\frac{g^l-1}{g-1}\right)+d_2\left(\frac{g^m-1}{g-1}\right)
  \end{equation}
  can have a solution. 
   The following applies.
\begin{enumerate}[i)]
 \item If case a) of Proposition \ref{prop:infinitely_sum} holds, there are infinitely many solutions.
 \item In all other cases, if a solution to \eqref{eq:lemma_lambda_1} exists, then
 \begin{equation}\label{eq:upper_n_lambda1_0}
 n<2.31\log g.
 \end{equation}
\end{enumerate}
\end{lemma}
\begin{proof} 
Let us assume $\Lambda_1 = 0$, which is equivalent to 
$$\left( b \pm 1 \right) b^n = \frac{d_2 g^m}{g-1}.$$ Since $b$ and $n$ are integers, and $g-1$ and $g$ are coprime, it follows that $g-1$ divides $d_2$. Given that $1 \leq d_2 \leq g-1$ and $g\geq 3$, this is possible only for $d_2=g-1$, and thus it follows that  $\left( b \pm 1 \right) b^n = g^m$. Substituting the above into the equations \eqref{eq:first_sum}, we obtain
$$1 \pm 1 = d_1 \left( \frac{g^l -1}{g-1} \right).$$
In the case $d_1 \left( \frac{g^l -1}{g-1} \right)=0$ we have a contradiction to the assumptions about $d_1, \, g$ and $l$. Hence, we have $d_1 \left( \frac{g^l -1}{g-1} \right)=2$.
Since $g \geq 3$, we see that $$ \frac{g^l -1}{g-1} = g^{l-1} + \ldots +1 > 2, $$ for every $l \geq 2$. Hence, it is left to consider the case $l=1$. Then we have $d_1=2$.

  Substituting $l=1, d_1=2$ and $d_2=g-1$ into \eqref{eq:lemma_lambda_1}, we get the equation 
  \begin{equation}\label{eq:zadnji_sl}
      \left( b \pm 1 \right) b^n = g^m.
  \end{equation}

  Let us first observe the case when $b\pm1=1$, i.e. $b=2$ and we have $b-1$ choice. Then 
  $g=2^k$ for some $k\in\mathbb{N}\setminus\{1\}$. There are infinitely many solutions described in Proposition \ref{prop:infinitely_sum} a).

If $b\pm1\geq 2$ from $\gcd(b,b\pm1)=1$ and \eqref{eq:zadnji_sl}, we know that
there exists a prime number $p$ that divides both $g$ and $b\pm 1$ but does not divide $b$. Then $p^m| b\pm 1$, hence $p^m\leq b\pm 1$ and
\begin{equation*}m\leq \frac{\log(b+1)}{\log 2}.\end{equation*}
Since $b^n|g^m$ we have 
$$n\leq m\frac{\log g}{\log b}\leq \frac{\log g\log (b+1)}{\log 2 \log b}<2.31\log g,$$
where we have used that $\log(b+1)/\log b<1.6$ for $b\geq 2$.
\end{proof}

Now, we consider the case when $\Lambda_1\neq 0$ and we want to apply Lemma \ref{tm:BMS} with 
$$\left( \gamma_1 , \gamma_2, \gamma_3 \right) \coloneqq \left( b, g, \frac{\left(b \pm 1 \right) \left(g-1 \right)}{d_2}  \right)$$ and $\left( b_1 , b_2, b_3 \right) \coloneqq \left( n, -m, 1 \right)$. Without loss of generality, we can assume $n>0$.
Note that $\gamma_1 , \gamma_2, \gamma_3 \in \mathbb{Q}$, thus, $D=1$.
Given that $h \left( \gamma_1 \right) =  \log{b}$ and $h \left( \gamma_2 \right) =  \log{g}$,
we can take $A_1 \coloneq \log{b}$ and $A_2 \coloneq \log{g}$. Moreover, as $\log(b+1) < 1.6\log b$ for all $b\geq 2$, it follows that
$$
h(\gamma_3) \leq \log(b+1) + \log(g-1) < 2.6 \log\left( \max\{b, g\} \right),
$$
and thus we may take
$$
A_3 \coloneqq 2.6 \log\left( \max\{b, g\} \right).
$$
Finally, given that $B\geq \max\{n,|-m|, 1\} $ and by applying Lemma \ref{lemma:g_3}, we can take $B \coloneqq 2.5 m \log{g}$.
Now, from inequality \eqref{lemma:g_5_3} and Lemma \ref{tm:BMS}, we derive
\begin{equation*}
    -C_1 \left( 1+ \log{\left( 2.5 m \log{g} \right)} \right) \log{b} \log{g} \log{ \left( \max{\{ b,g \}} \right) } 
<\log{\abs{\Lambda_1}}< - \left( m-l -2 \right) \log{g},
\end{equation*}
where $C_1=1.4 \cdot 30^6 \cdot 3^{4.5} \cdot 1^2 \cdot \left(1 + \log{1} \right) \cdot 2.6 < 3.723 \cdot 10^{11} $, from which it further follows
\begin{equation}\label{lemma:g_5_4}
    m-l < 3.73 \cdot 10^{11} \left( 1+ \log{\left( 2.5 m \log{g} \right)} \right) \log{b} \log{ \left( \max{\{ b,g \}} \right) }.
\end{equation}

\textit{Step 2.}
\\
After rearranging equations \eqref{eq:first_sum}, we obtain
\begin{equation}\label{lemma:g_5_5}
\left( b\pm 1 \right)b^n - g^l \frac{d_1 + d_2 g^{m-l}}{g-1} = \mp 1 - \frac{d_1 + d_2 }{g-1}.
\end{equation}
Taking the absolute value of \eqref{lemma:g_5_5} and using assumptions $d_1, d_2 \leq g-1$ we get
\begin{equation}\label{lemma:g_5_6}
\abs{\left( b\pm 1 \right)b^n - g^l \frac{d_1 + d_2 g^{m-l}}{g-1} } \leq 3 .
\end{equation}
Dividing the inequality \eqref{lemma:g_5_6} by $\left( b\pm 1 \right)b^n$ and using $b \geq 2$, we get
\begin{equation}\label{lemma:g_5_7}
\abs{ \frac{\left( d_1 +d_2 g^{m-l} \right) g^l}{\left( g-1 \right) \left( b \pm 1 \right)b^n} -1 }< \frac{1}{b^{n-2}}.
\end{equation}

We define
\begin{equation}\label{eq:lambda_2}
\Lambda_2 \coloneqq  \frac{ \left( d_1 +d_2 g^{m-l} \right) g^l }{ \left( g-1 \right) \left( b \pm 1 \right) b^n} -1 .
\end{equation}

To apply Lemma \ref{tm:BMS}, $\Lambda_2 \neq 0$ must hold. 
\begin{lemma}\label{lemma:lambda_2_neq_0}
If $\Lambda_2 = 0$, equations
\begin{equation}\label{eq:sum_pm_m}
b^n(b\pm1)-1=d_1\left(\frac{g^l-1}{g-1}\right)+d_2\left(\frac{g^m-1}{g-1}\right)
\end{equation}
can have a solution. If there is a solution for a choice of parameters $b$ and $g$, then $d_1+d_2=g-1$. The following applies.
\begin{enumerate}[i)]
 \item If any of the cases b), c), d) or e) of Proposition \ref{prop:infinitely_sum} hold, there are infinitely many solutions.
 \item In all other cases, if a solution exists, then
\begin{equation}\label{eq:upper_n_lambda2_0}
 n<\max\left\{5.39 \cdot 10^{11} \left( 1+ \log{\left( 2.5 m \log{g} \right)} \right) \log{g}  ,
2.69\frac{\log^3g}{\log^52}\right\}\log b\log(\max\{b,g\}).
 \end{equation}
\end{enumerate}
\end{lemma}
\begin{proof}
Let's assume $\Lambda_2=0$. This is equivalent to 
\begin{equation}\label{eq:l20_first_form}
\left( b \pm 1 \right) b^n = \frac{  d_1 g^{l} +d_2 g^{m}}{g-1}, 
\end{equation}
and by substituting it into \eqref{eq:first_sum}, we obtain $d_1+d_2=\mp(g-1)$.
In case $d_1+d_2 = -(g-1)$ we have a contradiction with assumptions about $d_1, d_2$ and $g$.

Let's consider the case $d_1 +d_2=g-1$. This equality occurs with equations \eqref{eq:sum_pm_m}. 

 Notice that $\Lambda_1 \neq 0$. If $\Lambda_1 = 0$, from Lemma \ref{lemma:lambda_1_0} we have that $d_2 = g - 1$, which would imply $d_1 = 0$, and that cannot hold.

If $m=l$, \eqref{eq:l20_first_form} becomes $$g^m=b^n(b\pm 1).$$
Let us first consider the case where $b\pm1=1$, i.e. $b=2$ and we have $b-1$ choice. Then $g^m=2^n$ and this implies $g=2^k$, for a positive integer $k$. There are infinitely many solutions for these parameters, as described in case b) of Proposition \ref{prop:infinitely_sum}.
If $b\pm1>1$, as in the proof of Lemma \ref{lemma:lambda_1_0}, we conclude that \eqref{eq:upper_n_lambda1_0} holds, i.e. 
\begin{equation}\label{upper_n_prva}
n<2.31\log g.
\end{equation}
Let $m>l$.
Notice that \eqref{eq:l20_first_form} can be rewritten as
\begin{equation}\label{eq:l20_second_form}
    b^n(b\pm 1)=g^l(1+d_2(g^{m-l-1}+\cdots+g+1)).
\end{equation}

Assume $b=g$ and $n\leq l$. After inserting in \eqref{eq:l20_second_form} and dividing by $b^n$, we can conclude that $b^{l-n}$ divides $b\pm 1$. Since $b\geq 2$ and $b\pm1$ are coprime, we conclude that $n=l$. This further implies 
$$b\pm 1=1+d_2(b^{m-l-1}+\cdots+b+1).$$
If $b+1$ is on the left-hand side, we would have $m-l-1=0$ and $d_2=b$, which cannot hold. If $b-1$ is on the left-hand side, we again have $m=l+1$ and $d_2=b-2=g-2$, which is described as case c) in Proposition \ref{prop:infinitely_sum} for $\alpha=\beta$. 

Assume $b=g$ and $n>l$. Then 
$$b^{n-l}(b\pm 1)=1+d_2(b^{m-l-1}+\cdots+b+1)$$
which implies that $b|(1+d_2)$, but $2\leq 1+d_2\leq b-1$, hence, this cannot hold.

It remains to consider the case $b\neq g$.

If $\gcd(b,g)=1$, from \eqref{eq:l20_second_form} we have $b^n|1+d_2(g^{m-l-1}+\cdots+g+1)$, implying
$$b^n\leq 1+d_2\cdot\left(\frac{g^{m-l}-1}{g-1}\right)< g^{m-l}.$$
 From \eqref{lemma:g_5_4} we conclude 
 \begin{equation}\label{upper_n_druga}
 n<(m-l)\frac{\log g}{\log b}<3.73 \cdot 10^{11} \left( 1+ \log{\left( 2.5 m \log{g} \right)} \right) \log{g} \log{ \left( \max{\{ b,g \}} \right) }.
 \end{equation}

If $\gcd(b,g)=d>1$, then there exist integers $b_1,g_1\geq 1$ such that $b=b_1d$ and $g=g_1d$ and $\gcd(b_1,g_1)=1$. 

Let us first assume $n< l$ and rewrite the equation in the form
$$b_1^n(b\pm 1)=d^{l-n}g_1^l(1+d_2(g^{m-l-1}+\cdots+g+1)).$$
If $g_1>1$, since $\gcd(b_1,g_1)=1$, then $g_1^l$ divides $b\pm 1$, implying 
\begin{equation}\label{upper_n_treca}
n< l<\frac{\log(b+1)}{\log 2}.
\end{equation}
If $g_1=1$, meaning that $g=d$ and $g|b$, we need another approach. If there exists a prime $p$ such that $p|b$ but $p\nmid g$ then 
$p^n\mid 1+d_2\frac{g^{m-l}-1}{g-1}$ which implies $2^n\leq p^n<g^{m-l}$. Together with  \eqref{lemma:g_5_4} we conclude 
\begin{equation}\label{upper_n_cetvrta}
n<(m-l)\frac{\log g}{\log 2}<5.39 \cdot 10^{11} \left( 1+ \log{\left( 2.5 m \log{g} \right)} \right) \log{b}\log{g} \log{ \left( \max{\{ b,g \}} \right) }.
\end{equation}

Hence, it is left to observe the special case where $g|b$ and $g$ and $b$ are products of the same prime factors.

 Since $g|b$, then $\gcd(g, b\pm1)=1$, and hence $g^l\mid b^n$. In addition, note that $g^l||b^n$. If that does not hold, i.e. if $g^{l+1}|b^n$ then $g|1+d_2$, which is in contradiction to $2\leq 1+d_2\leq g-1$. 
 
 Let $\alpha(p_i)$ and $\beta(p_i)$ denote exponents in the representation of the numbers $g$ and $b$ as a product of prime factors 
 $$g=\prod_{i=1}^w p_i^{\alpha(p_i)},\quad b=\prod_{i=1}^w p_i^{\beta(p_i)}.$$
 We have $\alpha(p_i)\leq \beta(p_i)$ for each $i$, and there exists at least one $p_i$ such that $\alpha(p_i)\neq \beta(p_i)$.  Also, we can approximate that for each $i$ we have $\alpha(p_i)\leq \log g/\log 2$ and $\beta(p_i)\leq \log b/\log 2$.
 
 If $g^l=b^n$, we would have 
 \begin{equation}\label{eq:bnjegl}
 (b\pm 1)(g-1)=d_1+d_2g^{m-l}.
 \end{equation}
 Since $g\mid b$, this implies $g\mid (d_1 \pm 1)$. Moreover, from $d_1+d_2=g - 1$ and $1\le d_1, d_2\leq g-1$ we obtain $1 \le d_1 \le g - 2$. Hence, this can occur only when $d_1 = 1$ and $g \mid (d_1 - 1)$.
  After inserting those values into the (\ref{eq:bnjegl}) and dividing by $g$, we observe  
 $$\frac{b}{g}(g-1)-1=(g-2)g^{m-l-1}.$$
 If $m-l-1\geq 1$, there would exist a prime $p$ that divides $b/g$ and $g^{m-l-1}$ which would imply $p|1$, a contradiction. But $m-l=1$ yields $b=g$, which does not hold in this case. 

 Hence, $g^l\neq b^n$. Let us observe 
 \begin{equation}\label{eq:special_case}
     \frac{b^n}{g^l}(b\pm 1)(g-1)=d_1+d_2g^{m-l}.
 \end{equation}
 There exists $p^k$, $k\in\mathbb{N}$, and $p$ a prime number such that $p^k||\frac{b^n}{g^l}$. Obviously, $p$ also divides $g$. Let $r\in\mathbb{N}$ be such that $p^r||g^{m-l}$. Then $p^{\min\{k,r\}}$ divides $d_1$. Hence,
 $$1\leq p^{\min\{k,r\}}\leq g-2.$$
 
  Let $A\subseteq \{1,2,\dots,w\}$, be the set of indices $i$ such that $k_i\geq 1$, where $k_i:=\beta(p_i)n-\alpha(p_i)l$. Since $b^n\neq g^l,$ we have $A\neq \emptyset$.
It holds
 $$2\leq \prod_{i\in A} p_i^{\min\{k_i,r_i\}}\leq g-2,\quad k_i,r_i\in\mathbb{N},$$
 where $r_i:=(m-l)\alpha(p_i).$ \\ If for some $i\in A$ we have $r_i\leq k_i$ then 
 $$2^{m-l}\leq p_i^{m-l}\leq p_i^{r_i}\leq g-2$$
 implying
 $$m-l\leq \frac{\log (g-2)}{\log 2}.$$
 From \eqref{eq:special_case} we see
  $$\frac{b^n}{g^l}\leq \frac{d_1+d_2g^{m-l}}{(b-1)(g-1)}\leq\frac{d_1}{g-1}+\frac{g-1-d_1}{g-1}g^{m-l}<g^{m-l},$$
 where we have used $g>2$, $b\geq 2$ and $d_1+d_2=g-1$. Hence,
 $$\frac{b^n}{g^l}=\prod_{i\in A} p_i^{k_i}<g^{\frac{\log(g-2)}{\log 2}}.$$
If for each $i$ we have $r_i>k_i$ then 
 $$\frac{b^n}{g^l}=\prod_{i\in A} p_i^{k_i}\leq g-2.$$
Since we observe values $g\geq 3$, it holds $g^{\frac{\log(g-2)}{\log 2}}\geq g-2$. Then, for each $i$, we can approximate $$k_i\leq {\log\left(g^{\frac{\log(g-2)}{\log 2}}\right)}/ {\log 2}=\frac{\log(g-2)\log g}{\log^22}<\frac{\log^2 g}{\log^2 2}.$$

 Let us observe a prime $p_i$ and its exponent $k_i$, for an $i\in A$.  It holds 
\begin{equation}\label{eq:linear_i}
\beta(p_i)n-\alpha(p_i)l=k_i.
\end{equation}
That means that $(n,l)$ is an integer pair solution of a linear Diophantine equation. This equation has infinitely many pairs of positive integer solutions, namely $(n,l)=\left(n_0+\frac{\alpha(p_i)}{d}\cdot t,l_0+\frac{\beta(p_i)}{d}\cdot t\right)$, $t\in\mathbb{N}_0$, and $d=\gcd(\alpha(p_i),\beta(p_i))$,  where $(n_0,l_0)$ denotes the pair of the smallest positive integers that are solutions of \eqref{eq:linear_i}. 
In order to find an upper bound on $n$, we need to find upper bounds on $n_0$ and $t$. 

 Notice, if $l>0$ then also $n>0$, since $k_i$, $\alpha(p_i)$ and $\beta(p_i)$ are positive integers. Since $(n_0,l_0)$ is a solution of \eqref{eq:linear_i} such that $l_0$ is the smallest integer with $l_0>0$, then $l_0$ is also the smallest positive solution of the congruence equation $-\alpha(p_i)l\equiv k_i \ (\bmod \ \beta(p_i))$, hence $l_0\leq \beta(p_i)$.
From \eqref{eq:linear_i}, we have
$$n_0=\frac{k_i+\alpha(p_i)l_0}{\beta(p_i)}\leq k_i+\alpha(p_i)< \frac{\log^2 g}{\log^2 2}+\frac{\log g}{\log 2}<1.64\frac{\log^2 g}{\log^2 2},$$
 where we have used that $l_0\leq \beta(p_i)$ and $g\geq 3$. 

First, assume that $w\geq 2$ and that there exists $j\neq i$, $j\in \{1,2,\dots,w\}$, such that $\beta(p_j)\alpha(p_i)-\alpha(p_j)\beta(p_i)\neq 0$. We insert solutions $(n,l)$ of \eqref{eq:linear_i} into the linear Diophantine equation that defines the exponent $k_j$  and observe
\begin{align*}
    k_j&=\beta(p_j)\left(n_0+\frac{\alpha(p_i)}{d}\cdot t\right)-\alpha(p_j)\left(l_0+\frac{\beta(p_i)}{d}\cdot t\right)\\
    &=(\beta(p_j)n_0-\alpha(p_j)l_0)+t\cdot(\beta(p_j)\alpha(p_i)-\alpha(p_j)\beta(p_i))/d.
\end{align*}
 If the number $\beta(p_j)\alpha(p_i)-\alpha(p_j)\beta(p_i)$ is a negative number, since $k_j\geq 0$ and $l_0>0$, we have 
\begin{equation}\label{eq:upper_t}
t<\beta(p_j)n_0d\leq\beta(p_j)n_0\beta(p_i)<1.64\frac{\log^2 g}{\log^2 2}\cdot \frac{\log^2 b}{\log^2 2},
\end{equation}
 where we have used $d\leq \beta(p_i)$.

If the number $\beta(p_j)\alpha(p_i)-\alpha(p_j)\beta(p_i)$ is a positive number, then $\beta(p_j)n_0-\alpha(p_j)l_0>0$ as well. If the contrary holds, $\beta(p_j)n_0-\alpha(p_j)l_0\leq 0$, we would have 
$$\frac{\alpha(p_j)\beta(p_i)}{\alpha(p_i)}<\beta(p_j)\leq \frac{l_0\cdot\alpha(p_j)}{n_0}$$
implying $k_i=\beta(p_i)n_0-\alpha(p_i)l_0<0$, a contradiction. Notice that in this case $k_j>0$ as well.
Now, from the expression for $k_j$, we notice that 
$$t\leq k_jd\leq k_j\beta(p_i)<\frac{\log^2 g}{\log^2 2}\cdot\frac{\log b}{\log 2}.$$
Hence, in both cases, the inequality \eqref{eq:upper_t} holds.
Then,
\begin{equation}\label{upper_n_peta}  
n=n_0+\frac{\alpha(p_i)}{d}\cdot t\leq n_0+\alpha(p_i)\cdot t<1.64\frac{\log^2 g}{\log^2 2}\cdot\left(1+\frac{\log g\log^2 b}{\log^3 2}\right)<2.69\frac{\log^3g\log^2 b}{\log^52},
\end{equation}
where we have used $g\geq 3$ and $b\geq 2$.

Let's now consider the case when $w\geq 2$ and for each $j\in\{1,2,\dots,w\}$, $j\neq i$, $\beta(p_j)\alpha(p_i)-\alpha(p_j)\beta(p_i)= 0$ holds. That means that $b$ and $g$ are powers of the same positive integer $a$ (exponents are proportional). 
Then the number $b^n/g^l$ is also a power of $a$. Since $g$ cannot divide $b^n/g^l$, and both numbers are powers of $a$, we have $b^n/g^l<g$ and $b^n/g^l$ divides $g$.  Observing \eqref{eq:special_case}, we see that then $b^n/g^l$ also divides $d_1$. Writing $d_1=z \cdot b^n/g^l$ for some integer $z \geq 1$ and dividing \eqref{eq:special_case} by $b^n/g^l$, we obtain
\begin{equation}\label{eq:poseban_slucaj_2}
(b \pm 1)(g - 1) = z + d_2 \cdot \frac{g}{b^n/g^l} g^{m-l-1}.
\end{equation}
We must have $\gcd(z, g) = 1$. If a prime $p$ divides both $z$ and $g$, it would also divide $a$. Consequently, the right-hand side of \eqref{eq:poseban_slucaj_2} would be divisible by $p$ as a multiple of $a$. This is impossible, as $p$ does not divide the left-hand side.
\\If $m-l>1$, we can reduce the equation \eqref{eq:poseban_slucaj_2} modulo $g$ and have $z\equiv \mp 1 \ (\bmod \ g)$ which is possible only for $z=1$ since $1\leq d_1\leq g-2$. Then $d_1=b^n/g^l$, and \eqref{eq:poseban_slucaj_2} becomes
$$(b- 1)g-b=d_2\cdot\dfrac{g^{l+1}}{b^n}g^{m-l-1}$$
which is divisible by $g$ and then we can observe
$$b-1-\frac{b}{g}=d_2\cdot\dfrac{g^{l+1}}{b^n}g^{m-l-2}.$$
The right-hand side is divisible by $a$, since $\frac{g^{l+1}}{b^n}$ is divisible by $a$. The left-hand side is not divisible by $a$, leading to a contradiction. \\
If $m-l=1$, then \eqref{eq:poseban_slucaj_2} yields
\begin{align*}
    (b\pm 1)(g-1)&=z+d_2\cdot \dfrac{g}{b^n/g^l}<\frac{g}{b^n/g^l}+\left(g-1-\frac{b^n}{g^l}z\right)\cdot\frac{g}{b^n/g^l}\\
    &=\frac{g}{b^n/g^l}\cdot \left(g-\frac{b^n}{g^l}z\right)\leq \frac{g}{2}(g-2)<(b-1)(g-1),
\end{align*}
a contradiction.

It only remains to consider the case $w = 1$. In this case, $b$ and $g$ are again powers of the same integer, a prime number $p$, which was considered in the previous case. 

Now assume $n>l$ and consider
$$d^{n-l}b_1^n(b\pm 1)=g_1^l(1+d_2(g^{m-l-1}+\cdots+g+1)).$$
If $b_1>1$, since $\gcd(b_1,g_1)=1$, we can conclude that 
$b_1^n<g^{m-l}.$
Hence
$n<(m-l)\frac{\log g}{\log2}$  and \eqref{upper_n_cetvrta} holds. If $b_1=1$, meaning $b|g$, we approach similarly as in the previous case. If there exists a prime number $p$ such that $p|g$ and $p\nmid b$, since $\gcd(b,b\pm1)=1$, then $p^l$ divides $b\pm1$, which implies $$l<\frac{\log(b+1)}{\log 2}.$$
Notice that $b^n<g^l\cdot g^{m-l}$. Hence
$$n<(l+(m-l))\frac{\log g}{\log b}<\left(\frac{\log (b+1)}{\log 2}+(m-l)\right)\frac{\log g}{\log b}.$$
After using \eqref{lemma:g_5_4} we get
 
 \begin{equation}\label{upper_n_sesta}
 n<3.74 \cdot 10^{11} \left( 1+ \log{\left( 2.5 m \log{g} \right)} \right) \log{g} \log{ \left( \max{\{ b,g \}} \right) }.
 \end{equation}
 
Similarly to before, we consider the special case where $b|g$ and $g$ and $b$ are products of the same prime factors. Since $\gcd(b,b\pm1)=1$, then $\gcd(g,b\pm 1)=1$, and again, we conclude that $g^l$ must divide $b^n$, and $g^{l+1}$ does not divide $b^n.$

If $g^l=b^n$ holds, then $g$ and $b$ must be powers of the same positive integer $a\geq 2$, implying $g=a^{\alpha}$ and $b=a^{\beta}$ for some positive integers $\alpha>\beta\geq 1$. Hence, $\beta n= \alpha l$.  
From \eqref{eq:bnjegl} and $b<g$ we have $(b\pm1)(g-1)<g^2$, while $d_1, d_2\ge1$ implies
$d_1+d_2g^{m-l} > g^{m-l}$. Hence $m-l\le1$, and since $m>l$, we have $m=l+1$, giving the families in cases c) and d) of Proposition \ref{prop:infinitely_sum} with $k=0$.

If $g^l\neq b^n$, we refer to \eqref{eq:special_case} and proceed analogously to the case $n<l$, since most of the arguments do not rely on that inequality or on its consequence $g<b$. Therefore, we discuss only the parts that require separate treatment.
 For an $i\in \{1,2,\dots,w\},$ such that $k_i\neq 0$, we observe the Diophantine equation \eqref{eq:linear_i}. If $w\geq 2$ and there exist $j\in\{1,2,\dots,w\}$, $j\neq i$, such that $\beta(p_j)\alpha(p_i)-\alpha(p_j)\beta(p_i)\neq 0$ we get the upper bound on $n$ similarly as \eqref{upper_n_peta}
 \begin{equation}\label{upper_n_sedma}
n<2.69\frac{\log^4g\log b}{\log^52},
 \end{equation}
 where we have used $d\leq \alpha(p_i)$.
\\
If $w=1$, or $w\geq 2$, but there does not exist any $j\neq i$ with the property $\beta(p_j)\alpha(p_i)-\alpha(p_j)\beta(p_i)\neq 0$, we have $b=a^{\beta}$, $g=a^{\alpha}$ for some positive integers $a$, $\alpha$, and $\beta$, and $a\geq2$ and $\alpha>\beta$ holds. 
Again, since $g$ cannot divide $b^n/g^l$ we have $b^n/g^l<g$ and $b^n/g^l$ divides $g$ since both are powers of $a$. Also, we have $d_1=b^n/{g^l}\cdot z$, where $\gcd(z,g)=1$. We observe
\eqref{eq:poseban_slucaj_2}. Notice that $b\leq g/2$ implies that $(b\pm 1)(g-1)<g^2$, hence, from \eqref{eq:poseban_slucaj_2} we conclude $m-l\leq 2$. \\
If $m-l=1$, \eqref{eq:l20_first_form} becomes 
$$(b\pm 1)b^n=g^l(1+d_2)$$
implying $d_2=\frac{b^n}{g^l}(b\pm1)-1$. Since $d_2\leq g-2$ we get $1\leq \beta n-\alpha l<\alpha-\beta$ if $d_2=\frac{b^n}{g^l}(b+1)-1$ and $1\leq \beta n-\alpha l\leq\alpha-\beta$ if $d_2=\frac{b^n}{g^l}(b-1)-1$. This yields infinitely many solutions, as described in cases c) and d) of Proposition \ref{prop:infinitely_sum} for parameter $k\neq 0$.
\\
If $m-l=2$, we observe the equation
\begin{equation}\label{eq:posebni}
g\left(b\pm 1-d_2\cdot \frac{g}{b^n/g^l}\right)=z+(b\pm 1).
\end{equation}
Notice that $b\leq g/a$ and 
$z=\frac{d_1}{b^n/g^l}\leq \frac{g-2}{a}.$ Then, for the right-hand side of (\ref{eq:posebni}), it holds
$$0<z+(b\pm 1)\leq \frac{2(g-1)}{a}+1\leq g,$$
where we have used $a\geq 2$. Since the left-hand side of (\ref{eq:posebni}) is a multiple of $g$, this is valid only when the right-hand side equals $g$. This is achieved for $a=2$, $b=g/2$, and $z=\frac{g-2}{2}$. Note that $b^n\neq g^l\geq 2$, which implies $d_1=b^n/{g^l}\cdot z=g-2$ and $b^n/ g^l=2$. After setting $b=2^k$ and $g=2^{k+1}$, for some positive integer $k$, and observing the Diophantine equation $kn-(k+1)l=1$, we obtain infinitely many solutions as described in case e) of Proposition \ref{prop:infinitely_sum}.

Finally, let $n=l$. Then 
$$b_1^n(b\pm 1)=g_1^n(1+d_2(g^{m-l-1}+\cdots+g+1)).$$
Since $b\neq g$, either $b_1>1$ or $g_1>1$ must hold.
If $b_1>1$, since $\gcd(b_1,g_1)=1$, we can conclude, as we did in the previous case, that 
$b_1^n<g^{m-l}.$
Hence
$n<(m-l)\frac{\log g}{\log2}$, i.e. \eqref{upper_n_cetvrta} holds.
If $g_1>1$, then $g_1^n\mid b\pm 1$, hence, \eqref{upper_n_treca} holds.

After comparing all upper bounds on $n$, namely \eqref{upper_n_prva}, \eqref{upper_n_druga}, \eqref{upper_n_treca}, \eqref{upper_n_cetvrta}, \eqref{upper_n_peta}, \eqref{upper_n_sesta}, and \eqref{upper_n_sedma}, we obtain the upper bound from the statement of the lemma.
\end{proof}

Let us consider the cases in which $\Lambda_2\neq 0$. We apply Lemma \ref{tm:BMS} with
$$\left( \gamma_1 , \gamma_2, \gamma_3 \right) \coloneqq \left( b, g, \frac{ d_1 +d_2 g^{m-l} }{ \left( g-1 \right) \left( b \pm 1 \right) } \right)$$ and $\left( b_1 , b_2, b_3 \right) \coloneqq \left( -n, l, 1 \right)$. Without loss of generality, we can assume $n>0$.
Note that $\gamma_1 , \gamma_2, \gamma_3 \in \mathbb{Q}$; thus, $D=1.$
Since $h \left( \gamma_1 \right) =  \log{b}, \, h \left( \gamma_2 \right) =  \log{g} $,
we can take $A_1 \coloneqq \log{b}, \, A_2 \coloneqq \log{g} $.
We have
\begin{multline*}
    h \left( \gamma_3 \right) \leq \max{ \left\{ \log{ \left( d_1 + d_2 g^{m-l} \right) }, \, \log{ \left( \left( g-1 \right) \left( b + 1 \right) \right) } \right\} } \\
    < \begin{cases}
			\log{g^{ m-l +2 }}, & \text{if $b \leq g$ or $b > g, \, m \neq l, \, b \leq g^{m-l}$},\\
            \log{b^2}, & \text{if $b > g, \, m \neq l, \, b \geq g^{m-l} $ or $b > g, m = l $},
		 \end{cases} 
\end{multline*}
thus
\begin{equation*}
    A_3 \coloneqq \begin{cases}
			\left( m-l +2 \right) \log{g}, & \text{if $b \leq g$ or $b > g, \, m \neq l, \, b \leq g^{m-l}$},\\
            2 \log{b} , & \text{if $b > g, \, m \neq l, \, b \geq g^{m-l} $ or $b > g, m = l $}.
		 \end{cases} 
\end{equation*}
Since $B\geq \max\{|-n|, l, 1\} $ and $l \leq m$, and by applying Lemma \ref{lemma:g_3}, we can take $B \coloneqq 2.5 m \log{g}$.

From inequality \eqref{lemma:g_5_7} and Lemma \ref{tm:BMS}, we obtain 
\begin{equation*}
    -C_2 \left( 1+ \log{\left( 2.5 m \log{g} \right)} \right) A_3 \log{b} \log{g}  < \log{\abs{\Lambda_2}}< -\left( n -2 \right) \log{b},
\end{equation*}
where $C_2=1.4 \cdot 30^6 \cdot 3^{4.5} \cdot 1^2 \cdot \left(1 + \log{1} \right) < 1.44 \cdot 10^{11}$, from which it further follows
\begin{equation}\label{lemma:g_5_8}
    n < 1.44 \cdot 10^{11} \left( 1+ \log{\left( 2.5 m \log{g} \right)} \right) A_3 \log{g} .
\end{equation}
From inequalities \eqref{lemma:g_5_4} and \eqref{lemma:g_5_8}, in all cases in which $\Lambda_1\neq0$ and $\Lambda_2\neq 0$, we have
\begin{equation}\label{lemma:g_5_9}
    n < 5.38 \cdot 10^{22} \left( 1+ \log{\left( 2.5 m \log{g} \right)} \right)^2  \log^2{g} \log{b} \log{ \left( \max{\{ b,g \}} \right) }.
\end{equation}

According to the Lemma \ref{lemma:g_4} and $b \geq 2, g \geq 3, n > 0$, it follows
\begin{equation}\label{lemma:g_5_10}
    1+ \log{\left( 2.5 m \log{g} \right) }< 15.81 \log{\left( n + 1.6 \right)} \log^{\frac{1}{2}}{b} \log^{\frac{1}{2}}{g}.
\end{equation}
From \eqref{lemma:g_5_9} and  \eqref{lemma:g_5_10}
we have
\begin{equation}\label{lemma:g_5_11}
    n + 1.6 < 1.35 \cdot 10^{25} \log^2{\left( n+ 1.6 \right)} \log^3{g} \log^2{b} \log{ \left( \max{\{ b,g \}} \right) }.
\end{equation}
In order to apply Lemma \ref{lemma:supporting} for \eqref{lemma:g_5_11} we define   
$$ L \coloneqq n+ 1.6, \, \ell \coloneqq 2 , \, H \coloneqq 1.35 \cdot 10^{25} \log^3{g} \log^2{b} \log{ \left( \max{\{ b,g \}} \right)} .$$
Note that $\ell >1$ and $H > (4\ell^2)^\ell= 2^8$. 
Thus, we have
\begin{equation}\label{lemma:g_5_12}
    n + 1.6 < 2^{2} \cdot 1.35 \cdot 10^{25} \log^3{g} \log^2{b} \log{ \left( \max{\{ b,g \}} \right) } \log^2{H}.
\end{equation}
Since $\max{\{ b,g \} } \geq 3$, we have
\begin{equation}\label{lemma:g_5_13}
    \log{H} < 62 \log^{\frac{1}{2}}{ \left( \max{\{ b,g \}} \right)}.
\end{equation}
From \eqref{lemma:g_5_12}  and \eqref{lemma:g_5_13} we obtain
\begin{equation}\label{lemma:g_5_16}
     n + 1.6 < 2^{2} \cdot 1.35 \cdot 10^{25} \cdot 62^2 \log^3{g} \log^2{b} \log^2{ \left( \max{\{ b,g \}} \right) },
\end{equation}
from which it follows that
\begin{equation}\label{lemma:g_5_15}
    n < 2.08 \cdot 10^{29} \log^3{g} \log^2{b} \log^2{ \left( \max{\{ b,g \}} \right) } .
\end{equation}

 After comparing with the upper bound on $n$ in the cases when $\Lambda_1=0$ or $\Lambda_2= 0$, namely  \eqref{eq:upper_n_lambda1_0} and \eqref{eq:upper_n_lambda2_0}, we conclude that the bound \eqref{lemma:g_5_15}  holds in all cases where there are finitely many solutions to the equations \eqref{eq:first_sum}.

From \eqref{lemma:g_5_15} and Lemma \ref{lemma:g_4} we get
\begin{equation}\label{lemma:g_5_14}
    l\leq  m < 2.71 \cdot 10^{29} \log^2{g} \log^3{b} \log^2{ \left( \max{\{ b,g \}} \right) }.
\end{equation}

\subsection{\texorpdfstring{Application for the $g=10$ and $2\leq b\leq 12$}{Application for the g=10 and 2 ≤ b ≤ 12}} In this subsection, our goal is to solve the equations given by \eqref{eq:first_sum} for certain fixed values of $g$ and $b$'s. A natural case to consider is when $g=10$, and we focus on examining the integers $b$ within the range of $2\leq b\leq 12$. Since $g=10$, it follows that $1\leq d_1,d_2 \leq 9$.

The proof is divided into two parts. Initially, we establish an effective numerical upper bound for $m-l$ by applying Lemma \ref{lemma:reduction}. Then, for each non-negative integer $m-l$ less than this bound, we determine an effective upper bound for $n$. This ensures that there are only a finite number of potential solutions for triples $(m,n,l)$, which can be explicitly verified. Some special cases are addressed separately.

Let us define 
\begin{equation}\label{eq:Gamma1_defn}
    \Gamma_1=\log(\Lambda_1+1)=n\log b-m\log 10+\log\frac{9(b\pm 1)}{d_2},
\end{equation}
where $\Lambda_1$ is defined in \eqref{lemma:g_5_3_0}. Then, from \eqref{lemma:g_5_3}, it holds
\begin{equation}
     \left|e^{\Gamma_1}-1\right|<\frac{1}{10^{m-l-2}}.
\end{equation}
We aim to improve a numerical upper bound on $m-l$. For practical purposes, we will assume from now on that $m-l\geq 3.$ Then 
$$\left|e^{\Gamma_1}-1\right|<\frac{1}{10}\implies \frac{9}{10}<e^{\Gamma_1}<\frac{11}{10}\implies \frac{10}{11}<e^{-\Gamma_1}<\frac{10}{9}.$$

 Assume that $\Gamma_1\neq 0$.
If $\Gamma_1>0$, we have
$$0<\Gamma_1<e^{\Gamma_1}-1=|e^{\Gamma_1}-1|<\frac{100}{10^{m-l}}.$$
If $\Gamma_1<0$, we have
$$0<|\Gamma_1|<e^{|\Gamma_1|}-1=e^{-\Gamma_1}(1-e^{\Gamma_1})<\frac{10}{9}\cdot\frac{100}{10^{m-l}}<\frac{112}{10^{m-l}}.$$
Hence, in each case we have 
\begin{equation}\label{eq:jdba_reduction}
    \left|n\frac{\log b}{\log 10}-m+\frac{\log({9(b\pm 1)}/{d_2})}{\log 10}\right|<\frac{112/\log 10}{10^{m-l}}.
\end{equation}
In order to apply Lemma \ref{lemma:reduction} to \eqref{eq:jdba_reduction}, the number $\tau = \frac{\log b}{\log 10}$ must be irrational, which holds true for $b \neq 10$. The case where $b = 10$ will be treated separately.
We also set parameters
$$w=m-l,\ \mu=\frac{\log({9(b\pm 1)}/{d_2})}{\log 10},\ A=112/\log 10,\ B=10.$$ 
From \eqref{lemma:g_5_15} we can take $M=2.08 \cdot 10^{29} \log^3{10} \log^2{b} \log^2{ \left( \max{\{ b,10 \}} \right) }$. 

 We implemented the algorithm of Lemma \ref{lemma:reduction} in Wolfram Mathematica, and for $1\leq d_2\leq 9$, $2\leq b\leq 12$, $b\neq 10$, and both choices $b+1$ and $b-1$ in the definition of $\mu$, we got the largest numerical bound for $w=m-l$, as shown in Table \ref{TAB_1}.
\begin{table}[h!]
\caption{\label{TAB_1} Upper bounds on $m-l$}
\begin{tabular}{|c|c|c|c|c|c|c|c|c|c|c|}
    \hline
    $b$ & $2$& $3$& $4$&$5$&$6$&$7$&$8$&$9$&$11$ &$12$  \\
    \hline 
    $m-l\leq$&$34$&$35$&$35$&$36$&$38$&$36$&$35$&$36$&$37$&$38$\\
    \hline
\end{tabular}
\end{table}
 
 In most cases, condition $\varepsilon > 0$ was satisfied, and the algorithm from part a) of Lemma \ref{lemma:reduction} yielded an upper bound on $m-l$. In cases where $\varepsilon < 0$, we applied part b) of Lemma \ref{lemma:reduction} and obtained an upper bound on $m-l$. However, there are cases where the solution to the linear congruence satisfies condition $n = n_0 < M$ and must be treated separately. Let us illustrate the details of our algorithm in some of these cases.
 
For $b = 2$, choosing $b-1$ in $\mu$ and $d_2 = 9$, we encountered $\varepsilon < 0$. This led to either $m-l \leq 34$ or $m-l \geq 35$ with $n = n_0 = 0$. Consequently, the inequality
$$
\left| 0 - m + \frac{\log\left(\frac{9}{9}\right)}{\log 10} \right| < \frac{112/\log 10}{10^{35}}
$$
holds. This implies $m = 0$, which results in a contradiction since $m \geq 35 + l \geq 36$.

For $b=5$, choice $b+1$ in $\mu$ and $d_2=7$ we got $\varepsilon<0$ and either $m-l\leq 35$ or $m-l\geq 36$ and $n=n_0=24739539326994274831296645391029$. 
Then the inequality
 $$\left|n_0\frac{\log 5}{\log 10}-m+\frac{\log({9\cdot 6}/{7})}{\log 10}\right|<\frac{112/\log 10}{10^{36}}$$
 holds. Number $m=17292195910660296022427834233064$ is the only integer solution to this inequality. This pair $(n,m)$ doesn't satisfy this inequality for exponent $m-l=37$ on the right-hand side, implying $m-l=36$ holds in this case.

If $\Gamma_1=0$, which implies $\Lambda_1=0$, from the proof of Lemma~\ref{lemma:lambda_1_0} we see that a solution exists only in the case $b=5$. In this situation we obtain $l=1$ and $m=n=2$. This satisfies the previously obtained bound $m-l=1<36$.

For the second part of the proof, we observe a linear form in logarithms 
$$\Gamma_2=\log(\Lambda_2+1)=l\cdot \log 10-n\log b+\log\frac{d_1+d_2\cdot 10^{m-l}}{9(b\pm 1)},$$
where $\Lambda_2$ is defined in \eqref{eq:lambda_2}.

Assume that $\Gamma_2\neq 0$.
From \eqref{lemma:g_5_7} we have $$\left|e^{\Gamma_2}-1\right|\leq \frac{1}{b^{n-2}}.$$

Without loss of generality, we assume that $n\geq 3$ and, similarly to before, we obtain 
\begin{equation}\label{eq:jdba_reduction_2}
    \left|l\frac{\log 10}{\log b}-n+\frac{\log[{(d_1+d_2\cdot 10^{m-l}})/{(9(b\pm 1))}]}{\log b}\right|<\frac{2/\log b}{b^{n-2}}.
\end{equation}
Again, for $b\neq 10$, number $\tau=\frac{\log 10}{\log b}$ is irrational. We apply Lemma \ref{lemma:reduction} with parameters 
$$w=n-2,\ \mu=\frac{\log[{(d_1+d_2\cdot 10^{m-l}})/{(9(b\pm 1))}]}{\log b},\ A=2/\log b,\ B=b,$$ 
 and $M=2.71 \cdot 10^{29} \log^2{10} \log^3{b} \log^2{ \left( \max{\{ b,10 \}} \right)  }$ from \eqref{lemma:g_5_14}. 

 We implemented the algorithm from Lemma \ref{lemma:reduction}, along with separate parts for the special cases explained above, in Wolfram Mathematica. Since $n$ appears on both sides of the inequality \eqref{eq:jdba_reduction_2}, all the special cases where $\varepsilon<0$ reduced to a contradiction. For $1\leq d_1, d_2\leq 9$, $2\leq b\leq 12$, $b\neq 10$, both choices $b+1$ and $b-1$ in the definition of $\mu$ and $0\leq m-l\leq 38$, we got the largest numerical bound for $w=n-2$  as shown in the Table \ref{TAB_2}.

\begin{table}[h!]
\caption{\label{TAB_2}Upper bounds on $n-2$}
\begin{tabular}{|c|c|c|c|c|c|c|c|c|c|c|c|}
    \hline
    $b$ & $2$& $3$& $4$&$5$&$6$&$7$&$8$&$9$&$11$ &$12$  \\
    \hline 
    $n-2\leq$&$119$&$77$&$61$&$52$&$49$&$45$&$40$&$38$&$36$&$33$\\
   
    \hline
\end{tabular}
\end{table}

By applying Lemma \ref{lemma:g_4}, we obtain upper bounds for the parameters $m$ and $l$, with the worst-case scenario being $l\leq m \leq55$. If $\Gamma_2 = 0$, then $\Lambda_2 = 0$. From Lemma~\ref{lemma:lambda_2_neq_0}, since we are considering $b \neq g = 10$, the equation \eqref{eq:l20_first_form} implies that $10^l$ divides $b^n(b \pm 1)$. By checking each possible value of~$b$, we find that for $b \in \{4,5,6,11\}$, we can have $l = 1$, and for $b = 5$, we also have $l = 2$. This yields $m \leq 38 + 2 = 40$. Hence, as shown before, we may take $l \leq m \leq 55$.

It is now easy to check all remaining possibilities $(b,n,l,m,d_1,d_2)$, $l\leq m$, and see if any of them are a solution to any equation from \eqref{eq:first_sum}. For $b=2$ we got that there are $71$ solutions of the form $(d_1,d_2,l,m,n)$, $l\leq m$. Among these, there are 5 solutions with parameters $(d_1, d_2, 1, 1, 0)$, 12 solutions with parameters $(d_1, d_2, 1, 1, 1)$, 20 solutions with parameters $(d_1, d_2, 1, 1, 2)$, and 14 solutions with parameters $(d_1, d_2, 1, 1, 3)$.

Note that for the given choice of $b$ and $g$, there are always multiple solutions of the form $(d_1,d_2,1,1,0)$ since we are solving equations $(b \pm 1) \pm 1 = d_1 + d_2$. These are the only solutions with $n = 0$ for $2 \leq b \leq 9$. For $b=11$, we also have one additional solution with $n=0$, namely $(2,1,1,2,0)$, and for $b=12$, we have 4 more solutions with $n=0$.

In the Table \ref{TAB_3}, for each $b$, we list the number of solutions $N$ with $n\geq 1$ as well as maximal values of $l, m$ and $n$. 

\begin{table}[h]
\caption{\label{TAB_3} Number of solutions and maximal values for indices}
\begin{tabular}{|c|c|c|c|c|c|c|c|c|c|c|c|}
    \hline
    
    $b$ & $2$& $3$& $4$&$5$&$6$&$7$&$8$&$9$&$11$ &$12$ \\
    \hline$N=$&$66$&$37$&$21$&$13$&$3$&$6$&$10$&$4$&$1$&$2$\\
    \hline
    $l\leq$&$2$&$2$&$2$&$2$&$1$&$2$&$2$&$1$&$2$&$2$\\
    \hline

    $m\leq$&$3$&$2$&$3$&$3$&$2$&$2$&$3$&$2$&$3$&$3$\\
    \hline

    $n\leq$&$8$&$3$&$3$&$3$&$1$&$1$&$2$&$1$&$1$&$1$\\
    \hline
    
\end{tabular}
\end{table}

Now we consider the case $b=g=10$. From Proposition \ref{prop:infinitely_sum}, we know that there are infinitely many solutions to the equation \eqref{eq:sum_mm} of the form $(1,8,n,n+1,n)$, where $n\in\mathbb{N}$.  
To identify other possible solutions, observe \eqref{eq:jdba_reduction} for $b=10$. More precisely, 
\begin{equation}\label{eq:b_g_10_first}
    \left|n-m+\frac{\log({9(10\pm 1)}/{d_2})}{\log 10}\right|<\frac{112/\log 10}{10^{m-l}}.
\end{equation}

Assume $m-l\geq 5$. For each choice of $d_2$, it is easy to see that there is no integer $n-m$ that satisfies \eqref{eq:b_g_10_first}. Hence, $m-l\leq 4$.

The next step is to observe \eqref{eq:jdba_reduction_2}, which simplifies to  
\begin{equation}\label{eq:b_g_10_second}
    \left|l-n+\frac{\log[{(d_1+d_2\cdot 10^{m-l}})/{(9(10\pm 1))}]}{\log 10}\right|<\frac{2/\log 10}{10^{n-2}}.
\end{equation}
For each possibility $0\leq m-l\leq 4$,  and $1\leq d_1,d_2\leq 9$, we easily get that $n\geq 3$ doesn't satisfy \eqref{eq:b_g_10_second} and $l,m\geq 1$. Hence $n\leq 2$. By using upper bounds as before, we obtained $34$ solutions of the form $(d_1,d_2,l,m,n)$, $l\leq m$, in addition to those described in Proposition \ref{prop:infinitely_sum}. Among these, there are $32$ solutions with parameters $(d_1, d_2, 1, 1, 0)$. The remaining $2$ solutions are $(1,1,1,2,0)$ and $(3,8,1,2,1)$.
We can summarize all the cases in the following result.

\begin{theorem}
Let $g=10$ and $2\le b\le 12$. 
Apart from the infinite family of solutions for $b=g=10$ given in Proposition \ref{prop:infinitely_sum}\, c), 
the Diophantine equations \eqref{eq:first_sum} have exactly $164$ solutions with $n\ge1$. Moreover, for $b=2,$ we have some representations as follows: 
\begin{align*}
(2-1)\cdot 2^8-1&=\dfrac{3\cdot (10^2-1)}{9}+\dfrac{2\cdot (10^3-1)}{9},\\
(2-1)\cdot 2^4+1&=\dfrac{9\cdot (10^1-1)}{9}+\dfrac{8\cdot (10^1-1)}{9},\\
(2+1)\cdot 2^5+1&=\dfrac{9\cdot (10^1-1)}{9}+\dfrac{8\cdot (10^2-1)}{9},\\
(2+1)\cdot 2^3-1&=\dfrac{1\cdot (10^1-1)}{9}+\dfrac{2\cdot (10^2-1)}{9}.
\end{align*}
\end{theorem}

\section{\texorpdfstring{The difference of two $g$-repdigits}{The difference of two g-repdigits}}

%
This section is devoted to proving Theorem \ref{tm:second}, so we consider equations \eqref{eq:second_diff}. Given that Lemma \ref{lemma:g_2} holds, we will assume $g \geq 3$ from now on.
%
\begin{lemma} \label{lemma:diff_m_l}
    All solutions to the equations \eqref{eq:second_diff} satisfy
    $$m \leq l .$$ 
\end{lemma}
\begin{proof}
    Since $b \geq 2$ and $n \geq 0$, it holds that $(b \pm 1) b^n \pm 1 \geq 0$, which is equivalent to
    \begin{equation} \label{eq:lem_diff_m_l_1}
        d_1 \left( \frac{g^l-1}{g-1} \right) \geq d_2 \left( \frac{g^m-1}{g-1} \right) .
    \end{equation}
    Since $1\leq d_1,d_2 \leq g-1$, from \eqref{eq:lem_diff_m_l_1} we have $g^l > g^{m-1}$. Since $g \geq 3$, the statement follows.
\end{proof}
\begin{lemma} \label{lemma:diff_l_n}
    All solutions to the equations \eqref{eq:second_diff} satisfy
    $$n < 1.5 l \log{g}.$$ 
\end{lemma}
\begin{proof}
    Given that $1 \leq  m$, $1\leq d_2$ and $d_1 \leq g-1$, from \eqref{eq:second_diff} we obtain 
    \begin{equation} \label{eq:lem_diff_n_n_1}
        (b\pm 1)b^n < g^l.
    \end{equation}
    Taking the logarithm of \eqref{eq:lem_diff_n_n_1}, we get 
    \begin{equation*}
        \log{\left(b \pm 1\right)} + n \log{b} < l \log{g},
    \end{equation*}
    from which, since $b \geq 2$, it follows that
\begin{equation} \label{eq:lem_diff_n_n_2}
        n \log{b} < l \log{g}.
    \end{equation}
    Dividing \eqref{eq:lem_diff_n_n_2} by $\log{b}$, and considering that $b \geq2$, the statement follows.
\end{proof}
\begin{lemma} \label{lemma:diff_m_l_n}
    Except for the cases described in Proposition \ref{prop:infinitely_diff}, all solutions to the equations \eqref{eq:second_diff} satisfy
    $$m \leq l < 1.3 (n+1.6) \frac{\log{b}}{\log{g}}+2.$$
\end{lemma}
\begin{proof}
  From Lemma \ref{lemma:diff_m_l}, it follows that $m \leq l$.
  In the case $m=l$, we have 
  \begin{equation} \label{eq:lem_diff_m_l_n_1}
      \left( b \pm 1 \right) b^n \pm 1 = \frac{g^l -1}{g-1} \left( d_1 - d_2 \right).
  \end{equation}
  Since $b\geq2$ and $n \geq 0$, it follows that $\left( b \pm 1 \right) b^n \pm 1 \geq 0$. Furthermore, since $g\geq3$, $l\geq1$, and $d_1, d_2 \geq 1$, we have that $d_1 - d_2 \geq 0$. If $d_1=d_2$, then \eqref{eq:lem_diff_m_l_n_1} implies $(b\pm1)b^n\pm1=0$,
which occurs only for case $(b-1)b^n-1$ with $b=2$ and $n=0$, as described in Proposition \ref{prop:infinitely_diff}. Therefore, we consider the case $d_1-d_2>0$.
  Notice, from equations \eqref{eq:lem_diff_m_l_n_1} and the fact that for all $b \geq 2$ and all $n \geq 0$ it holds that 
  $$\left( b \pm 1 \right) b^n \pm 1 < \left( \left( b + 1 \right) b^n \right)^{1.3},$$
  we get
  \begin{equation} \label{eq:lem_diff_m_l_n_2}
       g^{l-1} < \left( \left( b + 1 \right) b^n \right)^{1.3}.
  \end{equation}
  Taking the logarithm of inequality \eqref{eq:lem_diff_m_l_n_2} and using the fact that for all $b \geq 2$ it holds that $\log{(b+1)}<1.6 \log{b}$, we obtain
  \begin{equation} \label{eq:lem_diff_m_l_n_3}
      l < 1.3 (n+1.6) \frac{\log{b}}{\log{g}}+1.
  \end{equation}

Consider the cases when $ m < l$. Note that since $m \geq 1$, it follows that $l \geq 2$. Since $1 \leq d_1$ and $d_2 \leq g-1$, it follows from equations \eqref{eq:second_diff} that
 \begin{equation*}
      g^{l-1}+ \dots + 1 - g^m +1 \leq (b+1)b^n +1.
  \end{equation*}
In the case $m+1=l$, we have
\begin{equation} \label{eq:lem_diff_m_l_n_4}
       g^{l-2} < \left( b + 1 \right) b^n.
  \end{equation}  
  Taking the logarithm of inequality \eqref{eq:lem_diff_m_l_n_4} and using the fact that for all $b \geq 2$ it holds that $\log{(b+1)}<1.6 \log{b}$, we obtain
  \begin{equation} \label{eq:lem_diff_m_l_n_5}
      l < (n+1.6) \frac{\log{b}}{\log{g}}+2.
  \end{equation}
In the case $m+2 \leq l$, we have
\begin{equation} \label{eq:lem_diff_m_l_n_6}
       g^{l-1} < \left( b + 1 \right) b^n.
  \end{equation}
  Taking the logarithm of inequality \eqref{eq:lem_diff_m_l_n_6} and using the fact that for all $b \geq 2$ it holds that $\log{(b+1)}<1.6 \log{b}$, we obtain
  \begin{equation} \label{eq:lem_diff_m_l_n_7}
      l < (n+1.6) \frac{\log{b}}{\log{g}}+1.
  \end{equation}
  
  From \eqref{eq:lem_diff_m_l_n_3}, \eqref{eq:lem_diff_m_l_n_5}, \eqref{eq:lem_diff_m_l_n_7}, the statement follows.
\end{proof}
Now we are ready to prove Theorem \ref{tm:second}.

\smallskip 

\textit{Proof of the Theorem \ref{tm:second}}.
We will divide the proof into two parts, applying the Lemma \ref{tm:BMS} in each.

\textit{Step 1.}
\\
Multiplying the equations \eqref{eq:second_diff} by $g-1$ and performing simple rearrangements, we obtain
\begin{equation}\label{tm:second_1}
   \left( b \pm 1 \right) b^n \left( g-1\right) - d_1 g^l = \mp \left(g-1\right) - d_2 g^m - \left( d_1 - d_2 \right) .
\end{equation}
Taking the absolute value of \eqref{tm:second_1} and using assumptions $1 \leq d_1, d_2 \leq g-1, \; g \geq 3$ and $m \geq 1$, results in
\begin{equation}\label{tm:second_2}
  \abs{ \left( b \pm 1 \right) b^n \left( g-1\right) - d_1 g^l} < g^{m+2}. 
\end{equation}
Dividing the inequality \eqref{tm:second_2} by $d_1 g^l$ and using $d_1 \geq 1$, we get
\begin{equation}\label{tm:second_3}
\abs{ \frac{\left( b \pm 1 \right) b^n \left( g-1\right)}{d_1 g^l} -1} < \frac{1}{g^{l-m-2}}.
\end{equation}

We define 
\begin{equation}\label{tm:second_4}
\Lambda_3 \coloneqq \frac{\left( b \pm 1 \right) b^n \left( g-1\right)}{d_1 g^l} -1.
\end{equation}
Lemma \ref{tm:BMS} can be applied if $\Lambda_3 \neq 0$.

\begin{lemma} \label{lemma:diff_lambda3}
For all solutions to the equations \eqref{eq:second_diff}, it holds that
    $\Lambda_3 \neq 0$.
\end{lemma}
\begin{proof}
Assume the opposite, that there exists a solution $(d_1,d_2,l,m,n)$ to one of the equations \eqref{eq:second_diff} such that $\Lambda_3 = 0$, which is equivalent to
$$\left( b \pm 1 \right) b^n = \frac{d_1 g^l}{g-1}.$$
Since $b$ and $n$ are integers and $\gcd(g-1,g)=1$, it follows that $g-1 \vert d_1$. From $d_1 \leq g-1$, it further follows that
\begin{equation} \label{lemma:diff_lambda3_1}
    d_1=g-1,
\end{equation}
and we have
\begin{equation} \label{lemma:diff_lambda3_2}
   \left( b \pm 1 \right) b^n = g^l.
\end{equation}
Substituting \eqref{lemma:diff_lambda3_1} and \eqref{lemma:diff_lambda3_2} into equations \eqref{eq:second_diff}, we get
\begin{equation*}
    1\pm 1 = -d_2 \left( \frac{g^m -1}{g-1} \right).
\end{equation*}
Since $d_2, m \geq 1$ and $g \geq 3$, we have $d_2 \left( \frac{g^m-1}{g-1} \right) \geq 1$, which leads to a contradiction. 
\end{proof}

Now we can apply Lemma \ref{tm:BMS} with 
$$\left( \gamma_1 , \gamma_2, \gamma_3 \right) \coloneqq \left( b, g, \frac{\left(b \pm 1 \right) \left(g-1 \right)}{d_1}  \right)$$ 
and $\left( b_1 , b_2, b_3 \right) \coloneqq \left( n, -l, 1 \right)$. Without loss of generality, we can assume $n>0$.
Note that $\gamma_1 , \gamma_2, \gamma_3 \in \mathbb{Q}$, thus $D=1.$
Since $h \left( \gamma_1 \right) =  \log{b}, \, h \left( \gamma_2 \right) =  \log{g} $
and 
$$ h \left( \gamma_3 \right) \leq \log{(b+1)} + \log{(g-1)} < 2.6 \log{ \left( \max { \{ b,g \} } \right) } ,$$
we can take $A_1 \coloneq \log{b}, \, A_2 \coloneq \log{g}$ and 
$A_3 \coloneqq 2.6 \log{ \left( \max { \{ b,g \} } \right) }$.
Since $B\geq \max\{n,|-l|, 1\} $ and by applying Lemma \ref{lemma:diff_l_n}, we can take $B \coloneqq 1.5 l \log{g}$.

Now, from inequality \eqref{tm:second_3} and Lemma \ref{tm:BMS}, we derive
\begin{equation*}
    -C_1 \left( 1+ \log{\left( 1.5 l \log{g} \right)} \right) \log{b} \log{g} \log{ \left( \max{\{ b,g \}} \right) } < 
\log{\abs{\Lambda_3}}< - \left( l-m-2 \right) \log{g},
\end{equation*}
where $C_1=1.4 \cdot 30^6 \cdot 3^{4.5} \cdot 1^2 \cdot \left(1 + \log{1} \right) \cdot 2.6 < 3.723 \cdot 10^{11} $, from which it further follows
\begin{equation}\label{tm:second_5}
    l-m < 3.73 \cdot 10^{11} \left( 1+ \log{\left( 1.5 l \log{g} \right)} \right) \log{b} \log{ \left( \max{\{ b,g \}} \right) }.
\end{equation}

\textit{Step 2.}
\\
Rearranging equations \eqref{eq:second_diff}, we obtain
\begin{equation}\label{tm:second_6}
\left( b\pm 1 \right)b^n - g^m \frac{d_1g^{l-m} - d_2 }{g-1} = \mp 1 - \frac{d_1 - d_2 }{g-1}.
\end{equation}
Taking the absolute value of \eqref{tm:second_6} and using assumptions $1 \leq d_1, d_2 \leq g-1$ we get
\begin{equation}\label{tm:second_7}
\abs{\left( b\pm 1 \right)b^n - g^m \frac{d_1g^{l-m} - d_2 }{g-1} } < 2 .
\end{equation}
Dividing the inequality \eqref{tm:second_7} by $\left( b\pm 1 \right)b^n$ and using $b \geq 2$, we get
\begin{equation}\label{tm:second_8}
\abs{ \frac{\left( d_1g^{l-m} -d_2 \right) g^m}{\left( g-1 \right) \left( b \pm 1 \right)b^n} -1 }< \frac{1}{b^{n-1}}.
\end{equation}

We define
\begin{equation}\label{diff:lambda_4}
\Lambda_4 \coloneqq  \frac{ \left( d_1g^{l-m} -d_2 \right) g^m }{ \left( g-1 \right) \left( b \pm 1 \right) b^n} -1 .
\end{equation}

To apply Lemma \ref{tm:BMS}, $\Lambda_4 \neq 0$ must hold. 

\begin{lemma} \label{lemma:diff_lambda4}
For all solutions to the equations \eqref{eq:second_diff}, it holds that
    $\Lambda_4 \neq 0$.
\end{lemma}
\begin{proof}
Assume the opposite, that there exists a solution $(d_1,d_2,l,m,n)$ to one of the equations \eqref{eq:second_diff} such that $\Lambda_4 = 0$, which is equivalent to
\begin{equation} \label{lemma:diff_lambda4_1}
    \left( b \pm 1 \right) b^n = \frac{d_1 g^l-d_2g^m}{g-1}.
\end{equation}
Substituting \eqref{lemma:diff_lambda4_1} into equations \eqref{eq:second_diff}, we get
\begin{equation*}
    \pm (g-1)= d_2 - d_1,
\end{equation*}
but in both cases, we encounter a contradiction with $1 \leq d_1, d_2 \leq g-1$.
\end{proof}

Now we can apply Lemma \ref{tm:BMS} with
$$\left( \gamma_1 , \gamma_2, \gamma_3 \right) \coloneqq \left( b, g, \frac{ d_1g^{l-m} - d_2}{ \left( g-1 \right) \left( b \pm 1 \right) } \right)$$ and $\left( b_1 , b_2, b_3 \right) \coloneqq \left( -n, m, 1 \right)$. Except for the cases described in Proposition \ref{prop:infinitely_diff}, we have $d_1g^{l-m}-d_2>0$. Consequently, $\gamma_3>0$ and Lemma \ref{tm:BMS} applies. Also, we can assume $n>0$.
Note that $\gamma_1 , \gamma_2, \gamma_3 \in \mathbb{Q}$, thus $D=1.$
Since $h \left( \gamma_1 \right) =  \log{b}, \, h \left( \gamma_2 \right) =  \log{g} $,
we can take $A_1 \coloneqq \log{b}, \, A_2 \coloneqq \log{g} $.
We have
\begin{align*}
    h \left( \gamma_3 \right) &\leq \max{ \left\{ \log { \left( d_1g^{l-m} - d_2 \right)}, \, \log { \left( \left( g-1 \right) \left( b + 1 \right) \right) }   \right\} } \\
    &< \begin{cases}
			\log {  g^{ l-m +2 }  }, & \text{if $b \leq g$ or $b > g, \, m \neq l, \, b \leq g^{l-m}$},\\
           \log {  b^2 } , & \text{if $b > g, \, m \neq l, \, b \geq g^{l-m} $ or $b > g, m = l $},
		 \end{cases} 
\end{align*}
thus
\begin{equation*}
    A_3 \coloneqq \begin{cases}
			\left( l-m +2 \right) \log{g}, & \text{if $b \leq g$ or $b > g, \, m \neq l, \, b \leq g^{l-m}$},\\
            2 \log{b} , & \text{if $b > g, \, m \neq l, \, b \geq g^{l-m} $ or $b > g, m = l $}.
		 \end{cases} 
\end{equation*}
Since $B\geq \max\{|-n|, m, 1\} $ and $m \leq l$, and by applying Lemma \ref{lemma:diff_l_n}, we can take $B \coloneqq 1.5 l \log{g}$.

From inequality \eqref{tm:second_8} and Lemma \ref{tm:BMS}, we obtain 
\begin{equation*}
    -C_2 \left( 1+ \log{\left( 1.5 l \log{g} \right)} \right) A_3 \log{b} \log{g}  < \log{\abs{\Lambda_4}}< -\left( n -1 \right) \log{b},
\end{equation*}
where $C_2=1.4 \cdot 30^6 \cdot 3^{4.5} \cdot 1^2 \cdot \left(1 + \log{1} \right) < 1.44 \cdot 10^{11}$, from which it further follows
\begin{equation}\label{tm:second_9}
    n < 1.44 \cdot 10^{11} \left( 1+ \log{\left( 1.5 l \log{g} \right)} \right) A_3 \log{g} .
\end{equation}
From inequalities \eqref{tm:second_5} and \eqref{tm:second_9}, we have
\begin{equation}\label{tm:second_10}
    n < 5.38 \cdot 10^{22} \left( 1+ \log{\left( 1.5 l \log{g} \right)} \right)^2  \log^2{g} \log{b} \log{ \left( \max{\{ b,g \}} \right) }.
\end{equation}

According to the Lemma \ref{lemma:diff_m_l_n}, except in the case from Proposition \ref{prop:infinitely_diff}, since $b \geq 2$ and $g \geq 3$, it follows
\begin{equation}\label{tm:second_11}
    1+ \log{\left( 1.5 l \log{g} \right) }< 16.22 \log{\left( n + 1.6 \right)} \log^{\frac{1}{2}}{b} \log^{\frac{1}{2}}{g}.
\end{equation}
From \eqref{tm:second_10} and  \eqref{tm:second_11} we have
\begin{equation}\label{tm:second_12}
    n + 1.6 < 1.42 \cdot 10^{25} \log^2{\left( n+ 1.6 \right)} \log^3{g} \log^2{b} \log{ \left( \max{\{ b,g \}} \right) }.
\end{equation}
Let us apply Lemma \ref{lemma:supporting} for \eqref{tm:second_12}. Retaining the notations, we define   
$$ L \coloneqq n+ 1.6, \, \ell \coloneqq 2 , \, H \coloneqq 1.42 \cdot 10^{25} \log^3{g} \log^2{b} \log{ \left( \max{\{ b,g \}} \right)} .$$ Note that $\ell >1$ and $H > (4\ell^2)^\ell= 2^8.$ Thus, we have
\begin{equation}\label{tm:second_13}
    n + 1.6 < 2^{2} \cdot 1.42 \cdot 10^{25} \log^3{g} \log^2{b} \log{ \left( \max{\{ b,g \}} \right) } \log^2{H}.
\end{equation}
Since $\max{\{ b,g \} } \geq 3$, we have
\begin{equation}\label{tm:second_14}
    \log{H} < 62 \log^{\frac{1}{2}}{ \left( \max{\{ b,g \}} \right)}.
\end{equation}
From \eqref{tm:second_13}  and \eqref{tm:second_14} we obtain
\begin{equation}\label{tm:second_15}
     n + 1.6 < 2^{2} \cdot 1.42 \cdot 10^{25} \cdot 62^2 \log^3{g} \log^2{b} \log^2{ \left( \max{\{ b,g \}} \right) },
\end{equation}
from which it follows that
\begin{equation}\label{tm:second_16}
    n < 2.19 \cdot 10^{29} \log^3{g} \log^2{b} \log^2{ \left( \max{\{ b,g \}} \right) } .
\end{equation}
From \eqref{tm:second_13} and Lemma \ref{lemma:diff_m_l_n} we get
\begin{equation}\label{tm:second_17}
    m \leq l < 2.85 \cdot 10^{29} \log^2{g} \log^3{b} \log^2{ \left( \max{\{ b,g \}} \right) }.
\end{equation}

\subsection{\texorpdfstring{Application for the $g=10$ and $2\leq b\leq 12$}{Application for the g=10 and 2 ≤ b ≤ 12}}
Similarly as for equations \eqref{eq:first_sum}, in this subsection, we will describe all solutions of the Diophantine equation \eqref{eq:second_diff} for $g=10$ and $2\le b\le 12.$ For $b=2$, Proposition \ref{prop:infinitely_diff} describes infinitely many solutions to equation \eqref{eq:diff_mm}. For all other solutions, the upper bound on $l$ and $m$ in terms of $n$ from Lemma \ref{lemma:diff_m_l_n} holds.

Let
$$
\Gamma_3 := n\log b -l\log 10 +\log\left(\dfrac{9(b\pm 1)}{d_1}\right).
$$
Then, inequality \eqref{tm:second_3} can be written as
\begin{align*}
\abs{e^{\Gamma_3}-1}< \frac{1}{10^{l-m-2}}.
\end{align*}
Again, for practical purposes, we will assume from now on that $l-m\geq 3.$ Since $\Lambda_3=e^{\Gamma_3}-1\ne 0$, then $\Gamma_3\ne 0$. If $\Gamma_3> 0$, then we have the following inequalities:
\begin{align*}
0<\abs{\Gamma_3}=\Gamma_3<e^{\Gamma_3} -1= \abs{e^{\Gamma_3}-1}<\frac{1}{10^{l-m-2}}.
\end{align*}
If $\Gamma_3< 0$, then we have
\begin{align*}
1-e^{\Gamma_3}=\abs{e^{\Gamma_3}-1}<\frac{1}{10^{l-m-2}}<\frac{1}{2},
\end{align*}
$$
0<\abs{\Gamma_3} <  e^{\abs{\Gamma_3}}-1 = e^{\abs{\Gamma_3}}(1-e^{-\abs{\Gamma_3}})= e^{-\Gamma_3}\abs{1-e^{\Gamma_3}} < 2\cdot\frac{1}{10^{l-m-2}}.
$$
So, in all cases, whether $\Gamma_3< 0$ or $\Gamma_3> 0$ we have
$$ 
0<\abs{n\log b- l\log 10 +\log\left(\dfrac{9(b\pm 1)}{d_1}\right)}<\frac{2}{10^{l-m-2}}.
$$
Dividing by $\log 10$, we get
\begin{equation}\label{dp}
0<\abs{n\frac{\log b}{\log 10} -l + \frac{\log \left(9(b\pm 1)/d_1\right)}{\log 10}}<\frac{2/\log 10}{10^{l-m-2}}.
\end{equation}
 Since in this part we have to examine the cases $g=10$ and $2\le b\le 12,$ then from Theorem~\ref{tm:second}, we have the following estimate  $n<1.02\cdot 10^{32}.$
We may apply Lemma~\ref{lemma:reduction} to inequality \eqref{dp} with the following data $M=1.02\cdot 10^{32},$ $B=10,$
$$
w=l-m-2,\: \tau= \frac{\log b}{\log 10},\: \mu=\frac{\log \left(9(b\pm 1)/d_1\right)}{\log 10},\: \text{and} \: A=\frac{2}{\log 10}.
$$
For $g=b=10$, if we assume that $l-m-2\ge 3$, we would get a contradiction with inequality \eqref{dp}. Hence, we have $l-m\le 4.$ 

Note that if $b\ne 10$,  $\tau$ is an irrational number.  With Wolfram Mathematica, we apply Lemma~\ref{lemma:reduction} and obtain the data presented in Table \ref{TAB_4}. For every instance where $\varepsilon<0$, applying Lemma~\ref{lemma:reduction} b) yielded that $l-m$ is smaller than the largest bound calculated for the cases when $\varepsilon>0$. Hence, in all cases, we have $l-m\le 37.$ 

\begin{table}[h!]
\caption{\label{TAB_4} Upper bounds on $l-m-2$}
\begin{tabular}{|c|c|c|c|c|c|c|c|c|c|c|c|}
    \hline
    $b$ & $2$& $3$&$4$& $5$&$6$&$7$&$8$&$9$&$11$&$12$   \\
    \hline 
    $l-m-2\leq$&$34$&$33$&$34$&$35$&$34$&$34$&$34$&$33$&$35$& $34$\\
    \hline
\end{tabular}
\end{table}

The second part of the proof involves examining a linear form in logarithms $$\Gamma_4=\log(\Lambda_4+1)=m \log 10-n\log b+\log\frac{d_1\cdot 10^{l-m}-d_2}{9(b\pm 1)},$$
where $\Lambda_4$ is defined in \eqref{diff:lambda_4}. Moreover, we have 
$$
\left| e^{\Gamma_4}-1\right|<\dfrac{1}{b^{n-1}}.
$$
Then, by examining $\Gamma_4<0$ and $\Gamma_4>0$ in a similar way, we obtain that 
$$
0<\left| m \log 10-n\log b+\log\frac{d_1\cdot 10^{l-m}-d_2}{9(b\pm 1)}\right|<\dfrac{2}{b^{n-1}}
$$
is valid for $b\ge 2$ and therefore, we get
\begin{align}\label{ineq:diff_final_app}
0<\left| m \dfrac{\log 10}{\log b}-n+\dfrac{1}{\log b}\log\frac{d_1\cdot 10^{l-m}-d_2}{9(b\pm 1)}\right|<\dfrac{2/\log b}{b^{n-1}}.
\end{align}

In order to have a relatively acceptable upper bound of $n$ we need to apply Lemma~\ref{lemma:reduction} again using the following parameters
$M=1.02\cdot 10^{32},$ $B=b,$
$$
w=n-1, \quad \tau= \frac{\log 10}{\log b},\quad \mu=\frac{\log \left((d_1\cdot 10^{l-m}-d_2)/(9(b\pm1))\right)}{\log b} \ \text{and} \: A=\frac{2}{\log b}.
$$
Again, $\tau$ is irrational for any $b \neq 10$. By Lemma \ref{lemma:reduction}, it follows that $0 \le n \le 127$ across all such cases.

In Table \ref{TAB_5}, for each $b$, we list the number of solutions $N_0$ with $n=0$, the number of solutions $N$ with $n\geq 1$, as well as the maximal values for $l, m$ and $n$. Note that for $b=2$, care must be taken to exclude the solutions described in Proposition \ref{prop:infinitely_diff}, since some of them satisfy the bounds provided by Lemma \ref{lemma:diff_m_l_n}.

\begin{table}[h]
\caption{\label{TAB_5} Number of solutions and upper bounds on indices}
\begin{tabular}{|c|c|c|c|c|c|c|c|c|c|c|c|}
    \hline
    $b$ & $2$& $3$& $4$&$5$&$6$&$7$&$8$&$9$&$11$ &$12$ \\ 
    \hline$N_0=$&$22$&$27$&$24$&$20$&$16$&$12$&$9$&$13$&$18$&$4$\\
    \hline$N=$&$68$&$29$&$16$&$4$&$4$&$7$&$7$&$5$&$9$&$0$\\
    \hline
    $l\leq$&$4$&$3$&$3$&$2$&$2$&$2$&$2$&$3$&$3$&$3$\\
    \hline

    $m\leq$&$2$&$2$&$2$&$1$&$1$&$2$&$2$&$2$&$3$&$2$\\
    \hline

    $n\leq$&$10$&$4$&$4$&$1$&$1$&$1$&$1$&$2$&$1$&$0$\\
    \hline
\end{tabular}
\end{table}

When $b=10$, for each possibility where $0 \leq l-m \leq 5$ and $1 \leq d_1, d_2 \leq 9$, it is easy to determine that $n \geq 5$ does not satisfy the inequality \eqref{ineq:diff_final_app}, except for the cases $(d_1, d_2, l, m, n) = (9, 9, n+1, n, n)$ for choice $b-1$ and $(d_1, d_2, l, m, n) = (1, 1, n+2, n, n)$ for choice $b+1$. In these cases, it is straightforward to verify that they are not solutions to \eqref{eq:second_diff}. 
Thus, it is sufficient to consider $n \leq 4$. From Lemma \ref{lemma:diff_m_l_n}, we also know $m \leq l \leq 9$. By checking all these possibilities, we found $5$ solutions with $n=0$, $11$ solutions with $n=1$, and none for $n \geq 2$. 

We can summarize all the cases in the following result.
\begin{theorem}
In the range, $2\le b\le 12,$ the Diophantine equations \eqref{eq:second_diff} have exactly $160$ solutions with $n\ge1$. Moreover, for $b=2,$ we have some representations as follows: 
\begin{align*}
(2-1)\cdot 2^9-1&=\dfrac{5\cdot (10^3-1)}{9}-\dfrac{4\cdot (10^2-1)}{9},\\
(2-1)\cdot 2^6+1&=\dfrac{6\cdot (10^2-1)}{9}-\dfrac{1\cdot (10^1-1)}{9},\\
(2+1)\cdot 2^8+1&=\dfrac{7\cdot (10^3-1)}{9}-\dfrac{8\cdot (10^1-1)}{9},\\
(2+1)\cdot 2^2-1&=\dfrac{7\cdot (10^2-1)}{9}-\dfrac{6\cdot (10^2-1)}{9}.
\end{align*}
\end{theorem}

\section*{Acknowledgements}
The authors express appreciation to the referee for providing valuable feedback that enhanced the quality of this paper. The first  author is supported by IMSP, Institut de Math\'ematiques et de Sciences Physiques de l'Universit\'e d'Abomey Calavi. The second and third authors were partially supported by the project IP-UNIST-44 (ITPEM), which is funded by the NextGenerationEU Foundation.

\end{document}